\documentclass[11pt,oneside,reqno]{amsart}
\usepackage{amsmath,amssymb,amsthm}
\usepackage{geometry}                
\usepackage{graphicx}
\usepackage{amssymb}
\usepackage{epstopdf}
\usepackage{tikz}
\usepackage{enumerate}
\def\P{{\mathbb P}}
\def\N{{\mathbb N}}
\def\B{{\mathcal B}}

\def\sC{{\mathscr C}}

\def\e{{\varepsilon}}
\def\field{{K}}
\newcommand{\K}[1]{\left( #1 \right)}
\newcommand{\F}[2]{\frac{#1}{#2}}
\newcommand{\M}[2]{\left\{\left. #1 \; \right| \; #2 \right\}}
\newcommand{\sM}{M}
\newcommand{\A}{\alpha}
\newcommand{\sI}{\mathcal{I}}
\newcommand{\sS}{\mathcal{S}}
\newcommand{\sT}{\mathcal{T}}
\newcommand{\sA}{\mathcal{A}}
\newcommand{\sB}{\mathcal{B}}
\newcommand{\txt}[1]{\; \textrm{#1} \;}

\usepackage{fancyhdr, mathrsfs}

\def\set#1{\left\{ {#1} \right\}}
\def\setof#1#2{{\left\{#1\,:\,#2\right\}}}

\theoremstyle{plain}

\newtheorem{thm}{Theorem}[section]
\newtheorem{cor}[thm]{Corollary}
\newtheorem{prop}[thm]{Proposition}
\newtheorem{lem}[thm]{Lemma}

\theoremstyle{definition}

\newtheorem{defn}[thm]{Definition}

\newtheorem{notation}[thm]{Notation}
\newtheorem{rem}[thm]{Remark}

\title{Containment problem for points on a reducible conic in $\P^2$}

\author{Annika Denkert}
\email{andenken@gmx.de}
\address{Department of Mathematics, University of Nebraska, Lincoln, Nebraska 68588-0130}

\author{Mike Janssen}
\email{janssen@huskers.unl.edu}
\address{Department of Mathematics, University of Nebraska, Lincoln, Nebraska 68588-0130}
\thanks{The authors wish to thank Brian Harbourne and Susan Cooper for their aid and suggestions in this work, and the referee for his carefully considered and helpful comments.}

\begin{document}

\maketitle

\begin{abstract}
Given an ideal $I$ in a Noetherian ring, one can ask the containment question: for which $m$ and $r$ is the symbolic power $I^{(m)}$ contained in the ordinary power $I^r$? C. Bocci and B. Harbourne study the containment question in a geometric setting, where the ideal $I$ is in a polynomial ring over a field. Like them, we will consider special geometric constructs. In particular, we obtain a complete solution in two extreme cases of ideals of points on a pair of lines in $\P^2$; in one case, the number of points on each line is the same, while in the other all the points but one are on one of the lines.

%
%
\end{abstract}




\section{Introduction}

\subsection{Background}

Let $I \subseteq R = \field[\P^N]$ be a nontrivial homogeneous ideal.
If $I$ defines a set of points $p_1,p_2,\ldots,p_r \in \P^N$ (i.e., $I = \cap_i I(p_i)$, where $I(p_i)$ is the ideal generated by forms vanishing at $p_i$), then the $m$th symbolic power of $I$ is $I^{(m)} = \cap_i I(p_i)^m$.
Note that there is a more general definition of the symbolic power which is studied in \cite{ELS,HochsterHuneke1,HochsterHuneke-Fine}, among others.
It is not difficult to see that, if $I$ is the ideal of points in $\P^N$, we have $I^r
\subseteq I^{(r)} \subseteq I^{(m)}$ if and only if $r\geq m$; for the reverse containment, it is not difficult to see that $I^{(m)} \subseteq I^r$ implies $m\geq r$, but the converse is not true in general.
Using multiplier ideals and tight closure, respectively, \cite{ELS,HochsterHuneke1} proved, as a special case of a more general result, that, for a nontrivial homogeneous ideal $I\subseteq \field[\P^N]$ (where $\field$ is a field of arbitrary characteristic), $I^{(rN)} \subseteq I^r$.
In \cite{BocciHarbourne1,BocciHarbourne-Resurgence}, the question of when the symbolic power of an ideal $I$ is contained in an ordinary power is asked and answered in several cases; one such case is when $I$ is the ideal of points lying on a smooth conic in $\P^2$.
When the conic is not smooth (and hence consists of a pair of lines), the question of $I^r$ containing the symbolic power $I^{(m)}$ is more delicate, and depends on the number of points on each line, and the existence (or lack thereof) of a point at the intersection of the two lines.
A related problem, studied in \cite{BocciHarbourne1,BocciHarbourne-Resurgence}, is to compute an asymptotic quantity known as the resurgence:

\begin{defn}
Given a homogeneous ideal $I$ in $R = \field[\P^2] = \field[x,y,z]$, the resurgence, denoted $\rho(I)$, is the quantity:
\[
\rho(I) = \sup\M{m/r}{I^{(m)} \not\subseteq I^r}.
\]
\end{defn}

Recall that for non-trivial homogeneous ideals $I$ in $\field[\P^N]$, we have by \cite{ELS,HochsterHuneke1} that $I^{(rN)} \subseteq I^r$, and thus $\rho(I) \leq N$ always.
For particular ideals, however, sharper bounds and explicit computations of $\rho(I)$ in $\P^2$ are sometimes possible, though there is no known method of computing $\rho(I)$ that works in general.
In addition to computing $\rho(I)$ for ideals $I$ for two different configurations of points lying on a pair of lines, we affirmatively answer several questions of \cite{HarbourneHuneke,BCH} for the ideals defining our configurations.

\subsection{Preliminaries}

Throughout the remainder, $I$ is a nontrivial homogeneous ideal in $R = \field[\P^2] = \field[x,y,z]$, where $\field$ is a field of arbitrary characteristic.
Our primary goal is to give the best possible description of the set of all $m$ and $r$ for which $I^{(m)} \subseteq I^r$ if $I$ is a radical ideal defining either of the configurations of points in $\P^2$ found in Figure \ref{OurCases}.
Given a set of distinct points $p_1,p_2,\ldots,p_r\in \P^2$, we denote the scheme-theoretic union $Z$ of the points by $Z = p_1 + p_2 + \cdots + p_r$.
In order to more easily refer to these different situations in the future, we make the following definitions.

\begin{center}
\begin{figure}
\begin{tikzpicture}
\filldraw [black] (-5,0) circle (1pt)
(-4.5,0) circle (1pt)
(-4,0) circle (1pt)
(-3.5,0) circle (1pt)
(-4,1) circle (1pt)
(-3,0) circle (1pt);
\draw (-6,0) -- (-2,0);
\node[] (note2) at (-4,-1])
   {(a)};
\node[] (note3) at (5,-1])
   {(b)};
\filldraw [black] (5,0) circle (2pt)
(5.5,0) circle (1pt)
(6,0) circle (1pt)
(5.75,0) circle (1pt)
(6.25,0) circle (1pt)
(5,0.5) circle (1pt)
(5,0.75) circle (1pt)
(5,1) circle (1pt)
(5,1.25) circle (1pt);
\draw (4,0) -- (7,0);
\draw (5,-0.5) -- (5,1.75);
\draw (6.5,-0.25) -- (4.75,1.5);
\end{tikzpicture}
\caption{Almost collinear points (a); Nearly-complete intersection (b)}\label{OurCases}
\end{figure}
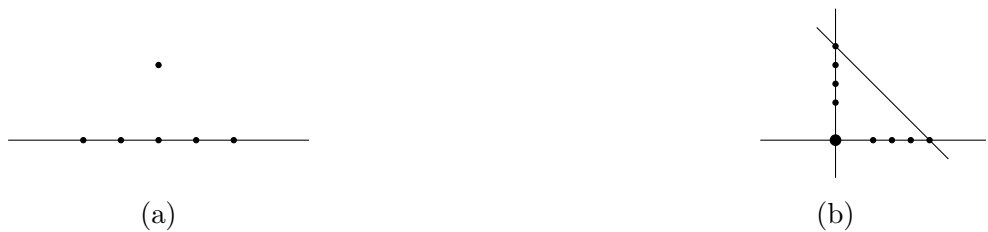
\end{center}
\vspace{-.2in}

\begin{defn}\label{DefnAC}
Let $Z = p_0 + p_1 + \cdots + p_n$ be a zero-dimensional subscheme
of $\P^2$, where $n \geq 2$. We call $Z$ an almost
collinear subscheme of $n+1$ points (or just an almost collinear
subscheme) if $p_1,p_2,\ldots,p_n$ lie on a line $L$ and $p_0$ does
not.
\end{defn}


\begin{defn}\label{DefnNCI}
Let $Z = p_0 + p_1 + p_2 + \cdots + p_{2n}$ be a zero-dimensional
subscheme of $\P^2$ with $n\geq 1$. We call $Z$ a nearly-complete intersection of
$2n+1$ points (or just a nearly-complete intersection) if there
exists a pair of lines $L_1$ and $L_2$ such that $p_0$ is the point
at the intersection of $L_1$ and $L_2$, $p_1,p_2,\ldots,p_n \in
L_1\setminus L_2$ and $p_{n+1},p_{n+2},\ldots,p_{2n} \in
L_2\setminus L_1$.
\end{defn}

The reason for the name given in Definition \ref{DefnAC} should be clear; without $p_0$, the configuration given in Definition \ref{DefnNCI} is a complete intersection.

\begin{rem}
A single point on a pair of lines is a complete intersection.
The ideal $I$ of a complete intersection is known to satisfy $I^m = I^{(m)}$ for all $m$ and hence $I^{(m)} \subseteq I^r$ if and only if $m\geq r$ (see Lemma 5 and Theorem 2 of Appendix 6 of \cite{Zariski-Samuel}).
Thus, we will not be interested in almost collinear subschemes $Z = p_0 + p_1 + \cdots + p_n$ when $n\leq 1$.
Moreover, the case that $n=2$ is by now well understood, and so will also not be of interest; see \cite{BocciHarbourne1,BocciHarbourne-Resurgence,BCH} for results in this case.
For the same reason, we will not consider nearly-complete intersections $Z = p_0 + p_1 + \cdots + p_{2n} \subseteq \P^2$ unless $n > 1$.
\end{rem}

Note that among reduced subschemes consisting of finitely many points on a pair of lines--but which are not complete intersections--the almost collinear case and the nearly-complete intersection case represent opposite extremes.
In the nearly-complete intersection situation, aside from the point at the origin, we have an equal number of points on each line.
In the almost collinear intersection situation, we are as far as possible from an equal number of points on each line (without being a set of collinear points).
Thus, it is not surprising that our results in the two cases are quite distinct, as indicated, for example, in Theorem \ref{ThmEarlyMainResult}, which shows that the solution for the almost collinear situation depends on the number of points on the line, whereas the solution for the nearly-complete intersection situation has no such dependency.

There is, however, an underlying similarity in both cases.
Whether $Z$ is an almost collinear subscheme or a nearly-complete intersection, $Z$ is the scheme-theoretic union of a complete intersection with a single point.
In both situations, we use this to find a vector space basis for $\field[x,y,z]$ which makes it easy to compare the symbolic and ordinary powers of the ideal $I(Z)$.



We can give a complete answer to the containment problem for almost collinear points and nearly-complete intersection:



\begin{thm}\label{ThmEarlyMainResult}
Let $I\subseteq \field[\P^2] = \field[x,y,z]$, where $\field$ is a field, be a homogeneous ideal of points.
Then:
\begin{enumerate}[(a)]
\item If $I$ defines $n+1$ almost collinear points, where $n\geq 3$, $\rho(I) = \frac{n^2}{n^2-n+1}$.
Moreover, $I^{(m)} \not\subseteq I^r$ holds if and only if  $m \leq \frac{n^2 r - n}{n^2-n+1}$.
%

\item If $I$ defines a nearly-complete intersection of $2n+1$ points, then $\rho(I) = \frac{4}{3}$.
Moreover, $I^{(m)}\subseteq I^r$ if and only if $4r\leq 3m+1$.
\end{enumerate}

\end{thm}

The proof of part (a) of Theorem \ref{ThmEarlyMainResult} will come in Theorems \ref{ThmMikeResurgence} and \ref{ThmMikeMainResult}.
The proof of part (b) of Theorem \ref{ThmEarlyMainResult} will come in Theorem \ref{CompleteSolution2}.

Our next main result guarantees that the symbolic power algebra $\oplus I^{(m)}$ is Noetherian (see Remark \ref{RemarkNoetherian}).

\begin{thm}\label{EarlyThmSymPowerAlg}
Let $I$ be a homogeneous ideal defining points in $\P^2$.
\begin{enumerate}[(a)]
\item If $I$ defines $n+1$ almost collinear points, where $n\geq 3$, then $I^{(nt)} = (I^{(n)})^t$ for every $t\geq 1$; moreover, $n$ is the least integer for which equality holds for all $t$.

\item If $I$ defines a nearly-complete intersection of $2n+1$ points, $I^{(2st)} = (I^{(2s)})^t$ for all $s,t\geq 1$.
\end{enumerate}

\end{thm}

The proof of part (a) of Theorem \ref{EarlyThmSymPowerAlg} is in Theorem \ref{PropSymPowersArePowers} and the proof of part (b) is Corollary \ref{CorSymPowersArePowers2}.
Also in Section \ref{SectionConsandApps}, we answer several questions of \cite{BCH,HarbourneHuneke} regarding containments of the form $I^{(m)}\subseteq \sM^i I^r$, where $\sM = (x,y,z)$ is the irrelevant maximal ideal.

\section{Main Results}\label{SecMainResults}

\subsection{Ideals of Almost Collinear Points}\label{SSecAlmostCP}

Let $\field$ be a field, and fix the ring $R = \field[\P^2] = \field[x,y,z]$.
The key to our proofs in both cases is to use use compatible $\field$-bases of $I^{(m)}$ and $I^r$, which we construct by first constructing a basis of $\field[x,y,z]$ and then restricting it to the ideals.

In particular, the following lemma is foundational to our approach.
Throughout, we use the notation $\langle S \rangle$ to denote the $\field$-span of the elements in the set $S$.

\begin{lem}\label{VSpaceBasisLemma}
Let $U$ and $V$ be subspaces of a vector space $W$.
Let $B_W$ be a basis of $W$ that contains a basis $B_U$ of $U$ and a basis $B_V$ of $V$.
Then $B_U\cap B_V$ is a basis for $U\cap V$.
\end{lem}

\begin{proof}
It is enough to show that $B_U \cap B_V$ spans $U\cap V$.
Suppose $a\in U\cap V$.
We know $a = \sum\limits_{e\in B_W} c_e e$ for $c_e\in \field$ (where $c_e = 0$ for all but finitely many $e$).
Since $a\in \langle B_U\rangle$, $c_e \not=0$ means $e\in B_U$.
Similarly, as $a\in \langle B_V\rangle$, $c_e \not=0$ implies $e\in B_V$.
Therefore, if $c_e\not=0$ we can conclude $e\in B_U\cap B_V$, and thus $a = \sum\limits_{e\in B_W} c_e e = \sum\limits_{e\in B_U\cap B_V} c_e e \in \langle B_U\cap B_V\rangle$.
\end{proof}

We first consider the case of almost collinear points; recall the definition of this configuration in Definition \ref{DefnAC}.
We use the following notation.

Let $Z$ be an almost collinear subscheme of $n+1$ points, and let $I = I(Z)$ be the ideal of forms vanishing at $Z$.
Assume that the collinear points $p_1, p_2,\ldots,p_n$ satisfy $z = 0$; specifically, let $p_1$ be defined by the intersection of the lines $x=0$ and $z=0$, and let $p_i$ be defined by lines $z=0$ and  $x-l_i y$, where $2\leq i\leq n$ and $l_i\not=0$.
Additionally, we may as well assume that $I(p_0)=(x,y)$.
Then this situation is described in Figure \ref{FigMikeCase}, and $I=(xz,yz,F)=(x,y)\cap (z,F)$, where $F = L_1 \cdots L_n$ is a homogeneous polynomial in $x$ and $y$ of degree $n$ (uniquely determined up to scalar multiple by the points $p_1,\ldots,p_n$), $L_1 = x$, and $L_i = x - l_i y$, where $p_i = (l_i, 1,0)$ for $2\leq i \leq n$.
Thus, we may assume that one term of $F$ is $x^n$.
With this setup, $I^{(m)} = (x,y)^{(m)} \cap (z,F)^{(m)} = (x,y)^m \cap (z,F)^m$, where the last equality follows from the fact that the ideals $(z,F)$ and $(x,y)$ define complete intersections.

\begin{figure}
\centering
\begin{tikzpicture}[scale=5/6]
\filldraw [black] (-1/2,0) circle (2pt) node[below] {$p_1$}
(1/4,0) circle (2pt) node[below] {$p_2$}
(1,0) circle (2pt) node[below] {$p_3$}
(7/4,0) circle (0pt) node[below] {$\cdots$}
(1,2) circle (2pt) node[below] {$p_0$}
(5/2,0) circle (2pt) node[below] {$p_n$};
\draw (-1,0) -- (3,0) node {\,\,\,\,\qquad $z=0$};
\draw (-2/3,-2/9) -- (4/3,22/9) node {\,\,\,\,\qquad $x=0$};
\draw (3/5,22/9) node {\!\!\!\!\!\!\!\!\!\!\!\!\!\!\!\!\!\! $y=0$} -- (3,-2/9);
\end{tikzpicture}
\caption{$n+1$ almost collinear points}\label{FigMikeCase}
\end{figure}
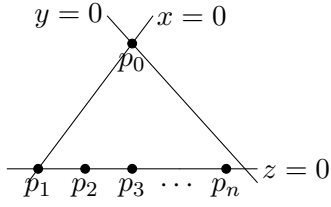

\begin{notation}
Given $F \in \field[x,y]$ of degree $n$ as above and $i$ a nonnegative integer, use the division algorithm to write $i=an+e$, where $0\leq e < n$.
For each $i$, we write $H_i : = x^e F^a$.
\end{notation}

Note that $\deg H_i = i$, and, as a polynomial in $x$, $H_i$ is monic with leading term $x^i$.
Moreover, $H_i \in (x,y)^i = (x,y)^{(i)}$.

\begin{lem}\label{LemXSpan}
Let $i \geq 0$.
Then $x^i$ is in the $\field$-span of $H_0 y^i, H_1 y^{i-1}, \dots, H_{i-1} y, H_i$.
\end{lem}
\begin{proof}
This is true for $i<n$, since $x^i = H_i$.
Suppose $i \geq n$, so $x^i\not= H_i$.
Then $H_i$ is a linear combinations of monomials of the form $x^t y^{i-t}$, where one term of $H_i$ is $x^i$.
Thus, $x^i - H_i$ is also linear combination of monomials of the form $x^t y^{i-t}$, where $t < i$ (as we have subtracted the $x^i$ term off).
By induction on $i$, each monomial $a_t x^t y^{i-t}$ appearing in the expansion of $x^i - H_i$ satisfies $a_t x^t y^{i-t} \in \langle H_0 y^i,\dots, H_{i-1} y \rangle$, and thus $x^i - H_i \in \langle H_0 y^i,\dots, H_{i-1} y\rangle$.
We conclude $x^i \in \langle H_0 y^i,\dots, H_i y^0\rangle$.
\end{proof}

\begin{lem}\label{RingBasis}
Consider $R = \field[\P^2] = \field[x,y,z]$.
A $\field$-basis of $R$ is given by $\B_R = \bigcup\limits_{i\geq 0} B_i$, where $B_i =  \setof{H_i y^j z^l}{i=an+e,\,0\leq e < n,\, H_i=x^e F^a, \text{ and } j,l\geq 0}$.
\end{lem}
\begin{proof}
By Lemma \ref{LemXSpan}, for each $t\geq 0$, $x^t$ is in the span of $H_0 y^t, \dots, H_t y^0$, hence every monomial $x^t y^s z^l$ is in the span of elements of the form $H_i y^j z^l$ with $i+j = t+s$.
Since the monomials of the form $x^t y^s z^l$ span $\field[x,y,z]$, so do the elements of the form $H_i y^j z^l$.

The elements $H_i y^j z^l$ are homogeneous and thus the span of those elements of degree $d$ must be the homogeneous component $R_d$ of $R = \field[x,y,z]$.
There are exactly ${d+2\choose 2} = \dim_{\field} R_d$ elements of the form $H_i y^j z^l$ of degree $d$ (since the cardinality of the set of those elements of the form $H_i y^j z^l$ is just the number of solutions $(i,j,l)$ to $i+j+l=d$ with $i,j,l\geq 0$).
Thus, the elements $H_i y^j z^l$ of degree $d$ are independent.
By homogeneity, any linear dependence among the elements of the form $H_i y^j z^l$ must involve elements of the same degree, hence $\B_R$ is linearly independent, and a $\field$-vector space basis of $R$.
\end{proof}

The next lemma places restrictions on $i,j,l$ which make elements of the form $H_i y^j z^l$ (with the restrictions) into a $\field$-basis of the symbolic power $I^{(m)}$.

\begin{lem}\label{SymPowerBasisLem}
Let $m \geq 1$.
\begin{enumerate}
\item[(a)] Then $H_i y^j z^l \in I^{(m)}$ if and only if $i,j,l\geq 0$, $i+ln\geq mn$, and $i+j\geq m$.
\item[(b)] Moreover, $I^{(m)}$ is the $\field$-vector space span of the elements of the form $H_i y^j z^l$ contained in $I^{(m)}$.
\end{enumerate}
\end{lem}

\begin{proof}
(a) Suppose $i,j,l\geq 0$, $i+ln\geq mn$, and $i+j\geq m$.
Then, since $i,j,l\geq 0$ and $i+j\geq m$, we have $H_i y^j z^l \in (x,y)^m$.
Since $i+ln\geq mn$, we have $i/n + l \geq m$, which is equivalent to $\lfloor i/n \rfloor + l \geq m$, which further implies $H_i y^j z^l \in (z,F)^m$.
Thus, $H_i y^j z^l \in (x,y)^m \cap (z,F)^m = I^{(m)}$.

Conversely, suppose $H_i y^j z^l \in I^{(m)}$.
Since $I^{(m)} = (x,y)^m \cap (z,F)^m$, we know $H_i y^j z^l \in (x,y)^m$, and thus $i+j \geq m$.
Also, $H_i y^j z^l \in (z,F)^m = (z,F)^{(m)}$, the order of vanishing of $H_i y^j z^l$ at $p_0$ must be at least $m$.
Since none of the points $p_1,\cdots,p_n$ are on the lines $x=0$ or $y=0$, $H_i y^j z^l \in (z,F)^{(m)}$ if and only if $F^b z^l \in (z,F)^{(m)}$, where $H_i = x^a F^b$.
But $F^b z^l \in (z,F)^{(m)}$ if and only if $b+l\geq m$, which holds if and only if $i+ln\geq mn$.

(b) Suppose we show that $(x,y)^m$ is the $\field$-vector space span of the elements of the form $H_i y^j z^l$ contained in $(x,y)^m$, and that $(z,F)^m$ is the $\field$-vector space span of the elements of the form $H_i y^j z^l$ contained in $(z,F)^m$.
Then, by Lemmas \ref{VSpaceBasisLemma} and \ref{RingBasis}, $I^{(m)} = (x,y)^m \cap (z,F)^m$ also is the $\field$-vector space span of the elements of the form $H_i y^j z^l$ contained in $I^{(m)}$.
Now, $(x,y)^m$ is the $\field$-span of monomials of the form $x^i y^j z^l$ with $i+j\geq m$, each of which is by Lemma \ref{LemXSpan} in the $\field$-span of elements of the form $H_i y^j z^l$ with $i+j\geq m$, each of which has order of vanishing at $p_0$ at least $m$ and hence is in $(x,y)^m$.
Finally, $(z,F)^m$ is the $\field$-span of elements of the form $x^t F^b y^s z^l$ with $b+l \geq m$.
But $x^t y^s$ is in $(x,y)^{t+s}$, and hence $x^t y^s$ is by Lemma \ref{LemXSpan} in the $\field$-span of elements of the form $H_q y^j$ with $q+j = t+s$, so each element $x^t F^b y^s z^l$ with $b+l\geq m$ is in the $\field$-span of elements of the form $H_i y^j z^l$ with $i=q+bn$, $q+j=t+s$ and $b+l\geq m$.
But $F^b z^l$ divides each $H_i y^j z^l$, and $F^b z^l \in (z,F)^m$ implies $H_i y^j z^l \in (z,F)^m$.
\end{proof}

We next provide a similar description of the elements of $I^r$, which will eventually allow us to completely answer the question of which lattice points $(m,r)$ correspond to containments $I^{(m)}\subseteq I^r$.

\begin{lem}\label{OrdinaryPowerBasisLemma}
Let $r\geq 1$.
\begin{enumerate}
\item[(a)] The ideal $I^r$ is the span of the elements of the form $H_i y^j z^l \in I^r$; in addition, if $H_i y^j z^l \in I^r$, then $H_i y^j z^l$ is a product of $r$ elements of $I$.
\item[(b)] Moreover, $H_i y^j z^l \in I^r$ if and only if $i,j,l\geq 0$ and either:
\begin{enumerate}
\item[(1)] $l < j$ and $i+nl \geq rn$, or
\item[(2)] $j\leq l < i+j$ and $i+j+(n-1)l \geq rn$, or
\item[(3)] $i+j\leq l$ and $r\leq i+j$.
\end{enumerate}
\end{enumerate}
\end{lem}
\begin{proof}
\textbf{Part (a):} This is true for $r=1$ by Lemma \ref{SymPowerBasisLem}(b).
Thus $I^r$ is the span of products $H_{i_1} y^{j_1} z^{l_1} \cdots H_{i_r} y^{j_r} z^{l_r}$ of $r$ elements of the form $H_{i_t} y^{j_t} z^{l_t}$, which satisfy $i_t, j_t, l_t \geq 0$, $i_t + l_t n \geq n$ and $i_t + j_t \geq 1$ for $t=1,\dots,r$ (i.e., elements of the form $H_{i_t} y^{j_t} z^{l_t} \in I$ for each $t$).

Write each $H_{i_t}$ as $x^{a_t} F^{b_t}$ where $i_t = b_t n + a_t$ and $0\leq a_t < n$.
Let $B = b_1 + \cdots + b_r$ and let $A = a_1 + \cdots + a_r$.
Then $H_{i_1} \cdots H_{i_r} = x^A F^B$ is, by Lemma \ref{LemXSpan}, in the span of elements of the form $H_u y^v F^B = H_{u+Bn} y^v$ where $u+v = A$ and $0\leq u\leq A$.

Since $i_1 + \cdots + i_r = (a_1 + \cdots + a_r) + n(b_1 + \cdots + b_r) = A + nB$, and since $H_u y^v$ is a product of $u+v = A = a_1 + \cdots + a_r$ linear forms, each of which is in $(x,y)$, we can factor $H_u y^v$ as $G_1 \cdots G_r$ where each $G_s$ is a product of $a_s$ of these linear forms.
Thus $H_{u+Bn} y^{v+j_1+\cdots+j_r} z^{l_1+\cdots+l_r} = (G_1 F^{b_1} y^{j_1} z^{l_1}) \cdots (G_r F^{b_r} y^{j_r} z^{l_r})$.
Now each $H_{i_t} y^{j_t} z^{l_t}$ satisfies $i_t,j_t,l_t \geq 0$, $i_t + l_t n \geq n$ and $i_t + j_t \geq 1$.
Thus $G_t  F^{b_t} y^{j_t} z^{l_t}$ satisfies $(a_t+b_t n) + l_t n = i_t + l_t n \geq n$ (thus either $b_t >0$ or $l_t > 0$ and so $G_t F^{b_t} y^{j_t} z^{l_t}$ vanishes at each point $p_1,\dots,p_n$) and $(a_t+b_t n) + j_t = i_t + j_t \geq 1$ (so $G_t F^{b_t} y^{j_t} z^{l_t}$ vanishes at $p_0$) and hence $G_t F^{b_t} y^{j_t} z^{l_t} \in I$.
Thus $H_{u+Bn} y^{v+j_1+\cdots+j_r} z^{l_1+\cdots+l_r} \in I^r$.

This shows not only that $I^r$ is the span of the elements of the form $H_i y^j z^l \in I^r$, but also that every element of $I^r$ is in the span of elements $H_i y^j z^l \in I^r$ which factor as a product of $r$ elements of $I$.
But if $H_i y^j z^l \in I^r$, it is in the span only of itself (since elements of this form are linearly independent), so each element $H_i y^j z^l \in I^r$ is itself a product of $r$ elements of $I$.

\textbf{Part (b):} Begin with the backward implication, and assume $i,j,l\geq 0$.

\begin{enumerate}
\item If $l < j$ and $i+nl \geq rn$, let $i = bn+a$, where $b = \lfloor i/n\rfloor$.
Then $l < j$ implies $F^b (yz)^l$ divides $H_i y^j z^l = x^a F^b y^j z^l$, but $i+nl \geq rn$ implies $b+l \geq r$, so $F^b (yz)^l$ is a product $r$ factors, each of which, being either $F$ or $yz$, is in $I$, hence $H_i y^j z^l \in I^r$.

\item If $j \leq l < i+j$ and $i+j+(n-1)l \geq rn$, then $l-j \geq 0$ and $i-(l-j) > 0$.
Let $t = \lfloor (i-(l-j))/n\rfloor$ and let $i = bn+a$, where $0 \leq a <n$. Note that $b = \lfloor i/n\rfloor \geq t$; let $G = x^a F^{b-t}$.
Then $H_i y^j z^l = x^a F^b (yz)^j z^{l-j} = GF^t (yz)^j z^{l-j}$, but $G \in (x,y)^{a+(b-t)n}$ and $a + (b-t)n = a+bn-nt \geq i-((i-(l-j))/n)n = l-j$.
Thus $H_i y^j z^l = F^t (Gz^{l-j})(yz)^j \in I^t I^{l-j} I^j = I^{t+l}$, but $(i-(l-j)) + nj + n(l-j) = i+j + (n-1)l \geq rn$ implies $(i-(l-j))/n + j + (l-j) \geq r$ and so $t+l \geq r$, whence $H_i y^j z^l \in I^{t+l} \subseteq I^r$.

\item Finally, if $r\leq i+j \leq l$, then $H_i y^j = G_1 \cdots G_r D$ where each $G_t$ is a linear form in $(x,y)$ and $D$ is a form in $(x,y)^d$ for $d = i+j-r$.
Thus $H_i y^j z^l = (G_1 z)\cdots(G_r z) (D z^{l-r})$, but $(G_1 z)\cdots(G_r z) \in I^r$, hence so is $H_i y^j z^l$.

\end{enumerate}

We now turn to the forward implication, but first a bit of terminology.
By \textit{minimal factor of $H_i y^j z^l$ in $I$} we mean a factor of $H_i y^j z^l$ which is in $I$ but which has no factor of smaller degree which is in $I$.
Minimal factors divisible by $z$ will be called $z$-factors.
Given any $H_i y^j z^l$, note that the minimal factors of $H_i y^j z^l$ in $I$ (if any) are of the form $F$, $yz$, $xz$, and $L_u z$ (where $L_u$ is the linear form vanishing on $p_0$ and on $p_u$ for some $1\leq u \leq n$).
Let $P_s$ denote a product of $s$ $z$-factors.
Any product $P_s F^t$ which divides $H_i y^j z^l$ satisfies $0\leq t \leq b$, where $b = \lfloor i/n\rfloor$, and $0 \leq s \leq \min\set{l,i+j-nt}$.
It is easy to see that if there are values for $s$ and $t$ satisfying $s+t \geq r$, $0\leq t \leq b$ and $0 \leq s \leq \min\set{l,i+j-nt}$, then $H_i y^j z^l$ has a factor $P_s F^t \in I^r$ and hence $H_i y^j z^l \in I^r$, while, by part (a), if $H_i y^j z^l \in I^r$ then $H_i y^j z^l$ has a factor $P_s F^t \in I^r$ with $s+t\geq r$ satisfying $0\leq t \leq b$ and $0\leq s \leq \min\set{l,i+j-nt}$.

Assume $H_i y^j z^l \in I^r$, and hence there are values for $s$ and $t$ satisfying $s+t \geq r$, $0\leq t\leq b$ and $0\leq s \leq \min\set{l,i+j-nt}$.
Of course, $i,j,l\geq 0$.
Then there are three cases: (a) $l < j$; (b) $j\leq l < i+j$; and (c) $i+j\leq l$.
\begin{enumerate}[(a)]
\item If $l < j$, then, since $i-nt\geq 0$, we have $\min\set{l,i+j-nt} = l$, so $r\leq t+s \leq b+l \leq i/n + l$, hence $i+ln \geq rn$.
This is case (1).

\item Suppose $j\leq l < i+j$.
If $l\leq i+j-nt$, then $s\leq \min\set{l,i+j-nt} = l$ and $t\leq (i+j-l)/n$, so $r\leq t+s \leq (i+j-l)/n+l$, or, equivalently, $nr \leq i+j + (n-1)l$ as we wanted to show.
If instead $l > i+j-nt$, let $\delta = l-(i+j-nt)$.
Then $s \leq \min\set{l,i+j-nt} = i+j-nt=l-\delta$, so $t= (i+j-l+\delta)/n$ and $r \leq t+s \leq (i+j-l+\delta)/n + l-\delta = (i+j+(n-1)l)/n-\delta(n-1)/n \leq (i+j+(n-1)l)/n$ which again implies $nr \leq i+j + (n-1)l$.
This is case (2).

\item If $i+j \leq l$, then $\min\set{l,i+j-nt} = i+j-nt$, so $r\leq s+t\leq i+j-(n-1)t\leq i+j$.
This is case (3).
\end{enumerate}
\end{proof}

We can now use Lemmas \ref{OrdinaryPowerBasisLemma} and \ref{SymPowerBasisLem} to compute the resurgence, $\rho(I)$.

\begin{thm}\label{ThmMikeResurgence}
For the ideal $I$ of $n+1$ almost collinear points,
\[
\rho(I) = \dfrac{n^2}{n^2-n+1}.
\]
\end{thm}
\begin{proof}
Consider $H_i y^j z^l$ where $i = tn^2$, $j=0$, and $l = tn^2-tn$, and let $m = tn^2$ and $r = tn^2-tn+t+1$.
Then $H_i y^j z^l \in I^{(m)}$ for every $t\geq 1$ by Lemma \ref{SymPowerBasisLem}(a), but $i + j + (n-1)l < rn$ so $I^{(m)}\not\subseteq I^r$ by Lemma \ref{OrdinaryPowerBasisLemma}(b)(2), hence $m/r \leq \rho(I)$ for all $t$.
Taking the limit as $t\to \infty$ gives $n^2/(n^2-n+1)\leq \rho(I)$.

Now suppose $m/r \geq n^2/(n^2-n+1)$ and hence $m\geq r$.
Consider $H_i y^j z^l \in I^{(m)}$.
Then $i+j \geq m$ and $i+nl \geq mn$ by Lemma \ref{SymPowerBasisLem}(a).
Now consider cases.

\begin{enumerate}[(a)]

\item If $l < j$, then $i+nl \geq mn \geq nr$ so $H_i y^j z^l \in I^r$ by Lemma \ref{OrdinaryPowerBasisLemma}(b)(1).

\item If $j \leq l < i+j$, use $i+j \geq m \geq rn^2/(n^2-n+1)$ and $i+nl \geq mn \geq rn^3/(n^2-n+1)$.
Arguing by contradiction, suppose that $i+j+(n-1)l < rn$.
Then $rn^2 > (n-1)i + i + nj + n(n-1)l = (n-1)(i+nl)+i+nj \geq rn^3(n-1)/(n^2-n+1)+i+nj$ so $rn^2(n^2-n+1) > rn^3(n-1) + (i+nj)(n^2-n+1)$ which simplifies to $rn^2 > (i+nj)(n^2-n+1)$.
Using $i+j \geq rn^2/(n^2-n+1)$, this gives $rn^2/(n^2-n+1) > i+nj \geq rn^2/(n^2-n+1)+(n-1)j$, which is impossible.
Thus $i+j+(n-1)l \geq rn$ so $H_i y^j z^l \in I^r$ by Lemma \ref{OrdinaryPowerBasisLemma}(b)(2).

\item If $i+j\leq l$, then $i+j\geq m \geq r$ so $H_i y^j z^l \in I^r$ by Lemma \ref{OrdinaryPowerBasisLemma}(b)(3).

\end{enumerate}

Thus $m/r \geq n^2/(n^2-n+1)$ implies $I^{(m)}\subseteq I^r$ by Lemma \ref{SymPowerBasisLem}(b), and so $\rho(I) \leq n^2/(n^2-n+1)$, i.e., $\rho(I) = n^2/(n^2-n+1)$.
\end{proof}

By definition, $\rho(I)$ is the supremum of rationals $m/r$ for which $I^{(m)}\not\subseteq I^r$, and thus it is possible to have $m/r \leq \rho(I)$ with $I^{(m)}\subseteq I^r$.
We next show that the bases found in previous lemmata allow us to completely answer the question of containment $I^{(m)}\subseteq I^r$ for all $m$ and $r$.

Containment will fail if and only if we can find $H_i y^j z^l \in I^{(m)}\setminus I^r$.
It is known that $I^{(m)} \not\subseteq I^r$ if $m < r$.
The constraints we have obtained show that if $m\geq r$, then $i+j \geq m$ and $i+nl \geq mn$ imply $i+j\geq r$ and $i+nl \geq rn$.
Thus, we have $H_i y^j z^l \in I^{(m)}\setminus I^r$ if and only if either
\begin{enumerate}
\item $m < r$, or
\item $m\geq r$, $i+j \geq m$ and $i+nl \geq mn$ (so $H_i y^j z^l \in I^{(m)}$), and $j \leq l < i+j$, $i+j + (n-1)l \leq rn-1$ (so $H_i y^j z^l \notin I^r$).
\end{enumerate}
If $m\geq r$, we have $H_i y^j z^l \in I^{(m)} \setminus I^r$ if and only if there is a non-negative integer lattice point $(i,j,l)$ satisfying $i+j\geq m$, $i+nl \geq mn$, $j\leq l \leq i+j-1$ and $i+j+(n-1)l\leq rn-1$.
In fact, we need only concern ourselves with $i$ and $l$, as the next lemma demonstrates.

\begin{lem}\label{LemmaCompleteSolution}
There is such a point $(i,j,l)$ if and only if there is a nonnegative integer lattice point $(i',l')$ satisfying $i' \geq m$, $i'+nl'\geq mn$, $l' < i'$ and $i' + (n-1)l' \leq rn-1$.
\end{lem}
\begin{proof}
Given $i'$ and $l'$, just take $i=i'$, $l=l'$, and $j=0$.
Given $(i,j,l)$, take $i' = i+j$ and $l' = l$.
\end{proof}

Therefore, $I^{(m)}\not\subseteq I^r$ if and only if either $m < r$ or there is a nonnegative integer lattice point $(i,l)$ satisfying
\begin{equation}\label{CompleteSolutionConstraints}
i\geq m,\quad i+nl\geq mn,\quad l\leq i-1,\quad\text{and}\quad i+(n-1)l\leq rn-1.
\end{equation}

\begin{thm}\label{ThmMikeMainResult}
Let $I$ be the ideal of $n+1$ almost collinear points and $m\geq r$ integers.
Then $I^{(m)} \not\subseteq I^r$  if and only if $m\leq \dfrac{n^2 r - n}{n^2-n+1}$. 
\end{thm}
\begin{proof}
Let $P$ be the point $(i,l)$ where the lines $i+nl = mn$ and $i+(n-1)l = rn-1$ cross; i.e., $P = (mn-n^2(m-r)-n,n(m-r)+1)$.
Let $Q$ be the point where the lines $l=i-1$ and $i+nl=mn$ cross; i.e., $Q = (n(m+1)/(n+1),(nm-1)/(n+1))$.
Let $U$ be the point where the lines $m=i$ and $i+nl=mn$ cross; i.e., $U = (m,m(n-1)/n)$.
Then (\ref{CompleteSolutionConstraints}) has a solution if and only if the $i$-coordinate of $P$ is at least as big as the maximum of the $i$-coordinates of $Q$ and $U$. Let $Q_i$ and $U_i$ be these $i$-coordinates; then $\max(Q_i,U_i) = Q_i$ if $m\leq n$, while $\max(Q_i,U_i) = U_i$ if $m\geq n$.

Thus, assuming $m\geq r$, (\ref{CompleteSolutionConstraints}) has a solution if and only if either $m\leq n$ and $Q_i \leq P_i$, or $m\geq n$ and $U_i \leq P_i$.
But $Q_i \leq P_i$ is the same as $n(m+1)/(n+1)\leq mn-n^2(m-r)-n$ or $m \leq r(n+1)/n-(n+2)/n^2 = (rn^2+rn-n-2)/n^2$, and $U_i \leq P_i$ is the same as $m\leq mn-n^2(m-r)-n$ or $m\leq (n^2r-n)/(n^2-n+1)$.

Thus, $I^{(m)}\not\subseteq I^r$ holds if and only if either
\begin{enumerate}[(a)]
\item $m < r$, or
\item $m\geq r$ and either
\begin{enumerate}[(i)]
\item $m\leq n$ and $m\leq (rn^2+rn-n-2)/n^2$, or
\item $m\geq n$ and $m\leq (n^2r-n)/(n^2-n+1)$.
\end{enumerate}
\end{enumerate}
Note, however, that if $1 \leq m < r$, then $r\geq 2$ and so $m\leq (rn^2+rn-n-2)/n^2$ holds (since $r\leq (rn^2+rn-n-2)/n^2$ if $r\geq 2$ and $n\geq 3$), and also $m\leq (n^2r -n)/(n^2-n+1)$ (since $r\leq (rn^2+rn-n-2)/n^2$ if $r\geq 2$ and $n\geq 3$).
Thus $m < r$ is subsumed by $m\leq n$ and $m\leq (rn^2+rn-n-2)/n^2$, or $m\geq n$ and $m\leq (n^2r-n)/(n^2-n+1)$.

However, we can do even better by ridding ourselves of the need for the two cases $m < n$ and $m\geq n$.

\textbf{Claim:} We have $I^{(m)}\not\subseteq I^r$ if and only if $m\leq \frac{n^2 r - n}{n^2-n+1}$.

\textbf{Proof of Claim. }
If $m<n$ and $m\leq (rn^2+rn-n-2)/n^2$, then routine arithmetic demonstrates $m\leq r$.
Now, if $I^{(m)}\not\subseteq I^r$ then we already know that either $m<n$ and $m\leq \frac{rn^2+rn-n-2}{n^2}$ or $m\geq n$ and $m\leq \frac{rn^2-n}{n^2-n+1}$.
If $m<n$ and $m\leq \frac{rn^2+rn-n-2}{n^2}$, then we now know that $m\leq r$, but $I^{(m)}\not\subset I^r$ implies $r > 1$, and, as we are assuming $n\geq 3$, it follows that $r\leq \frac{rn^2-n}{n^2-n+1}$, and hence $m\leq (rn^2-n)/(n^2-n+1)$.
Conversely, assume $m\leq \frac{rn^2-n}{n^2-n+1}$.
If $m\geq n$, then we already know that $I^{(m)}\not\subseteq I^r$, so assume $m<n$.
If $m < r$, then $I^{(m)}\not\subseteq I^r$, so we may also assume $r \leq m$.
So either $m=r$ or $r+1\leq m\leq \frac{rn^2-n}{n^2-n+1}$.
If $r+1\leq  m\leq \frac{rn^2-n}{n^2-n+1}$, then routine arithmetic shows that $n < \frac{n^2+1}{n-1} \leq r \leq m$, which contradicts $m < n$.
Thus we must have $m=r < n$. But $m=r=1$ is impossible since $m=r=1$ implies $1\leq \frac{n^2-n}{n^2-n+1}$, which is false, so we must have $1 < m=r < n$.
More arithmetic demonstrates that $m\leq \frac{rn^2+rn-n-2}{n^2}$ which we have already showed implies $I^{(m)}\not\subseteq I^r$.
\end{proof}

These initial containment results for almost collinear points stand
in contrast to the results obtained in the next section for
nearly-complete intersections. In particular, the results for almost
collinear points depend on the number $n$ of points on the line,
whereas the results we obtain for nearly-complete intersections do
not.

\subsection{Ideals of Nearly-Complete Intersections}\label{SSecNCI}

Let $R=\field[x,y,z]$ and $n\in\N$.
Suppose we have $n$ points on $L_1$, defined by $x=0$, say $p_1,\ldots, p_n$, and $n$ points on $L_2$, defined by $y=0$, say $p_{n+1}, \ldots, p_{2n}$, all of multiplicity $m$.
We assume that there is one additional point $p_0$ of multiplicity $m$ at the intersection of $L_1$ and $L_2$, as in Figure \ref{FigAnnikaCase}.

\begin{figure}[h]
\begin{center}
\begin{tikzpicture}
\filldraw [black] (0,0) circle (2.5pt) (1,0) circle (1.5pt)
(3/2,0)circle (1.5pt) (2,0) circle (1.5pt) (5/2,0) circle (1.5pt)
(0,1) circle (1.5pt) (0,3/2) circle (1.5pt) (0,2) circle (1.5pt)
(0,5/2) circle (1.5pt); \draw (-1,0) -- (3.5,0); \draw (1,0)
node[below] {$p_1$}; \draw (1.8,-0.1) node[below] {$\ldots$}; \draw (0,0)
node[below left] {$p_0$}; \draw
(2.5,0) node[below] {$p_n$}; \draw (3,0) node[above right] {$x=0$};
\draw (0,-0.5) -- (0,3.5); \draw (-0,2.5) node[left] {$p_{2n}$};
\draw (-0.3,1.8) node[left] {$\vdots$}; \draw (0,1) node[left]
{$p_{n+1}$} ; \draw (0,3) node[right] {$y=0$}; \draw (3,-0.5) --
(-1,3.5); \draw (-1,3) node[above left] {$z=0$};
\end{tikzpicture}
\end{center}
\caption{$2n+1$ points of a nearly-complete intersection}\label{FigAnnikaCase}
\end{figure}

The ideal defining these $2n+1$ points is
$$I =(x,y)\cap \K{\bigcap_{i=1}^n(x,z-\A_i y)\cap \bigcap_{i=1}^n
(y,z-\beta_i x)},$$ where $\A_i, \beta_i \in \field$, $\sI(p_i)=(x,z-\A_i
y)$ for $i=1,\ldots,n$, and $\sI(p_i)=(y,z-\beta_i x)$ for
$i=n+1,\ldots,2n$. Then for any $m\in\N$,
$$I^{(m)}=(x,y)^m\cap\bigcap_{i=1}^{2n}\sI(p_i)^m.$$

Define a polynomial $F\in R$ by
$$F=z^n-\prod_{i=1}^n(z-\beta_i x)-\prod_{i=1}^n(z-\A_i y).$$

\begin{prop}\label{SymbolicAsIntersection}
Let $I$ and $F$ be as defined above. Then
$I=(x,y)\cap(xy,F)=(xy,xF,yF)$ and for any $m\in\N$,
$I^{(m)}=(x,y)^m\cap(xy,F)^m$.
\end{prop}

\begin{proof}
Let $I':=\bigcap_{i=1}^n(x,z-\A_i y)\cap\bigcap_{i=1}^n (y,z-\beta_i
x)$ and consider the two curves $C_1: xy=0$ and $C_2: F=0$. Then
$C_1$ and $C_2$ intersect exactly at the $2n$ points
$p_1,\ldots,p_{2n}$, and transversely at that. Therefore $I'$ is
generated by the forms defining $C_1$ and $C_2$, hence $I'=(xy,F)$
and therefore $I=(x,y)\cap (xy,F)$.

Also, note that $(x,y)\cap(xy,F)\supseteq (xy,xF,yF)$.

Suppose $g\in (x,y)\cap (xy,F)$, say $g=k_1(xy)+k_2F$ with $k_1,k_2\in R$.
Since $g\in (x,y)$, we have
\begin{eqnarray*} 0 &=& g([0,0,1])\\
    &=& (k_1(xy)+k_2F)([0,0,1])\\
    &=& 0+k_2([0,0,1])\cdot F([0,0,1])\\
    &=& k_2([0,0,1]) \cdot(-1)\\
    &=& -k_2([0,0,1])
\end{eqnarray*}
by definition of $F$.
Thus $k_2([0,0,1])=0$ and hence $k_2\in (x,y)$.
But then $g\in (xy,(x,y)F)=(xy,xF,yF)$ and hence $I=(x,y)\cap(xy,F)= (xy,xF,yF)$.

Since the ideals $(x,y)$ and $(xy,F)$ are complete intersections,
they are each generated by a regular sequence. Therefore, by Lemma 5 and Theorem 2 of Appendix 6 of \cite{Zariski-Samuel}, symbolic
and ordinary powers of the ideals coincide.
Thus $I^{(m)}=(x,y)^{(m)}\cap (xy,F)^{(m)}=(x,y)^m\cap (xy,F)^m$.
\end{proof}

As in Subsection \ref{SSecAlmostCP}, we will use a vector space basis to describe $I^{(m)}$ and $I^r$.

\begin{prop}\label{LemXSpan2}
Let $t\geq 0$.
Then every monomial $x^ay^bz^c$ with
$a,b,c\in\N_0$ and $c\leq t$ is contained in the $\field$-span $\sS$ of
$\M{x^ay^bz^cF^d}{a,b,c,d\in\N_0, c<n, dn\leq t}$.
\end{prop}

\begin{proof} Use induction on $t$.
For $t<n$, the condition $dn\leq t$ implies that $d=0$, so $$\sS=\langle\M{x^ay^bz^c}{a,b,c\in\N_0, c<n}\rangle.$$
In particular, $x^ay^bz^c$ with $c\leq t<n$ is in $\sS$.

For $t\geq n$, assume that $x^ay^bz^c\in\sS$ for every $a,b\in\N_0$
and $c< t$.
Take $a,b\in\N_0$ and consider the polynomial $G=x^ay^bz^t-x^ay^bz^r(-F)^q$ with $t=qn+r$ and $0\leq r<n$.
Then by the definition of $F$, $G$ is a polynomial of degree less than $t$ in $z$ with coefficients in $\field[x,y]$.
Since $qn\leq t$, we have $G\in\sS$ by induction.
But $r<n$, so $x^ay^bz^r(-F)^q=(-1)^qx^ay^bz^rF^q\in\sS$ and hence $x^ay^bz^t=G+x^ay^bz^r(-F)^q\in\sS$.
\end{proof}

\begin{lem}\label{RingBasis2}
The set $\sA=\M{x^ay^bz^cF^d}{a,b,c,d\in\N_0, c<n}$ is linearly
independent over $\field$ and spans $R$ as a $\field$-vector space.
\end{lem}

\begin{proof}
By Proposition \ref{LemXSpan2}, each monomial $x^ay^bz^c$ is in $\langle\sA\rangle$, and since $$R=\langle\M{x^ay^bz^c}{a,b,c\in\N_0}\rangle,$$ we get that $R=\langle\sA \rangle$.
For each $s\in\N_0$, define a subset $\sA_s$ of $\sA$ by
$$\sA_s:=\M{x^ay^bz^cF^d}{a,b,c,d\in\N_0, c<n, a+b+c+dn=s}.$$
Then $\sA=\bigsqcup_{s=0}^{\infty}\sA_s$ and the elements of $\sA_s$ are
homogeneous of degree $s$, and therefore $\langle\sA_s\rangle=R_s$.
By homogeneity of $\sA_s$, all elements in $\sA_s$ are linearly independent from elements in $\sA\setminus \sA_s$, and $|\sA_s|={s+(3-1)\choose 3-1}=\dim R_s$ (the number of partitions of $s$ into three parts in the non-negative integers), which means that $\sA_s$ not only spans $R_s$ but also has the same size as a (linearly independent) monomial basis for $R_s$.
Therefore, $\sA_s$ is linearly independent as well, and so is $\sA$.
\end{proof}

\begin{prop}\label{SymPowerBasisProp}
Let $m\in\N$. Then
\begin{enumerate}
    \item The set $\sB=\M{x^ay^bz^cF^d}{a,b,c,d\in\N_0, c<n, a+b\geq m}$ is linearly independent over $\field$ and spans $(x,y)^m$ as a $\field$-vector space.
    \item The set $\sC=\M{x^ay^bz^cF^d}{a,b,c,d\in\N_0, c<n, \min(a,b)+d\geq m}$ is linearly independent over $\field$ and spans $(xy,F)^m$ as a $\field$-vector space.
\end{enumerate}
Therefore, $$I^{(m)}=\langle x^ay^bz^cF^d | a,b,c,d\in\N_0, c<n, a+b\geq m, \min(a,b)+d\geq m\rangle.$$
\end{prop}

\begin{proof} Since $\sB, \sC\subseteq \sA$ from Lemma \ref{RingBasis2}, linear independence of $\sB$ and $\sC$ over $\field$ is immediate.

By definition, Proposition \ref{LemXSpan2}, and Lemma \ref{RingBasis2},
\begin{eqnarray*} (x,y)^m &=& (x^ay^b|a,b\in\N_0, a+b=m) \\
    &=& \sum_{a+b=m}x^ay^bR\\
    &\subseteq& \langle x^ay^bz^cF^d| a,b,c,d\in\N_0, c<n, a+b\geq m
    \rangle\\
    &\subseteq& (x,y)^m
\end{eqnarray*}
and likewise
\begin{eqnarray*} (xy,F)^m &=& (x^ay^aF^d|a,d\in\N_0, a+d=m)\\
    &=& \langle x^ay^bz^cF^d| a,b,c,d\in\N_0, c<n, \min(a,b)+d\geq m
    \rangle .
\end{eqnarray*}
Hence $I^{(m)}=(x,y)^m\cap(xy,F)^m=\langle \sB\rangle \cap \langle\sC\rangle$, and Lemma \ref{VSpaceBasisLemma} then gives $I^{(m)}=\langle\sB\cap\sC\rangle= \langle x^ay^bz^cF^d | a,b,c,d\in\N_0, c<n, a+b\geq m, \min(a,b)+d\geq m\rangle$.
\end{proof}

We will describe $I^r$ in a similar fashion.

\begin{lem}\label{OrdinaryPowerBasisLemma2Aux1}
Let $r\in\N$.
Define one set $\sS_r$ by $$\sS_r=\M{(a,b,c,d)\in\N_0^4}{c<n,\min(a,b)+d\geq r, a+b+d\geq 2r, \txt{and} a+b\geq r},$$ and another set, $\sT_r$, by
$$\displaystyle \sT_r=\{(a,b,c,d)\in\N_0^4 | c<n, (\min(a,b)+d\geq r \txt{if} d\leq \max(a,b)-\min(a,b)),$$ $$\txt{} (a+b+d\geq 2r \txt{if}\max(a,b)-\min(a,b)<d<b+a), \txt{and} (a+b\geq r \txt{if} a+b\leq d)\}.$$
Then $\sS_r=\sT_r$.
\end{lem}

\begin{proof} If $(a,b,c,d)\in \sS_r$, then is it also always in $\sT_r$, hence $\sS_r\subseteq \sT_r$. 
Now suppose $(a,b,c,d)\in\sT_r$.
We may assume without loss of generality that $a\leq b$, so $\min(a,b)=a$ and $\max(a,b)=b$.

(a) If $d\leq b-a$ and consequently $a+d=\min(a,b)+d\geq r$, then $b\geq d+a\geq r$ and hence $a+b\geq r$ and $(a+d)+b\geq 2r$.
Therefore, $(a,b,c,d)\in\sS_r$.

(b) If $b-a<d<b+a$ and consequently $a+b+d\geq 2r$, then $2(a+b)>(a+b)+d\geq 2r$, so $a+b\geq r$.

(c) If $b>r$, then $\min(a,b)+d=a+d>(b-a)+a=b>r$, and if $b\leq r$, then $r-b\geq 0$ and so $\min(a,b)+d =a+d\geq 2r-b=r+(r-b)\geq r$.
Therefore, $(a,b,c,d)\in\sS_r$.

(d) Finally, if $a+b\leq d$ and consequently $a+b\geq r$, then $d\geq a+b\geq r$ and hence $\min(a,b)+d\geq d\geq r$ and $(a+b)+d\geq 2r$.

Therefore, $(a,b,c,d)\in\sS_r$.
\end{proof}

\begin{lem}\label{OrdinaryPowerBasisLemma2Aux2}
Let $r\in\N$ and $\sS_r$ and $\sT_r$ be as above.
Then $$J:=\langle x^ay^bz^cF^d | (a,b,c,d)\in\sS_r \rangle= \langle x^ay^bz^cF^d | (a,b,c,d)\in\sT_r\rangle$$ is an ideal.
\end{lem}

\begin{proof}
Note that the definition of $\sS_r$ immediately gives that
$(a,b,c,d)\in\sS_r$ implies $(a+a', b+b', c, d+d')\in\sS_r$ for all $a',
b', d'\in \N_0$, because $(a+a')+(b+b')\geq a+b\geq r$, $\min(a+a',
b+b')+(d+d')\geq \min(a,b)+d\geq r$, and $(a+a')+(b+b')+(d+d')\geq
a+b+d\geq 2r$. First we will show that $x^ay^bz^{c+c'}F^d\in J$
for all $(a,b,c,d)\in\sS_r$ and all $c'\in\N_0$.

To see this, induct on $c'$. Take $(a,b,c,d)\in\sS_r$ and let
$h=x^ay^bz^{c+c'}F^d$. If $c'<n-c$, then $(a,b,c+c',d)\in\sS_r$ and
hence $h\in J$. If $c'\geq n-c$, then $c'+c\geq n$, say $c'+c=qn+p$
with $p<n$. Assume that $x^ay^bz^{c+t}F^d\in J$ for all $t\leq
c'-1$. By definition of $F$, $L:=F+z^n$, considered as a polynomial
in $z$ with coefficients in $K[x,y]$, has degree at most $n-1$. Thus
$h=x^ay^bz^pF^d(z^n)^q=x^ay^bz^pF^d(L-F)^q=
x^ay^bz^pF^d\sum_i\eta_iL^{q-i}F^i=\sum_i\eta_ix^{a+a_i}y^{b+b_i}
z^{p+c_i}F^{d+i}$ for some $\eta_i\in K$ and $a_i, b_i, c_i\in\N_0$.
But $p+c_i<c'$ for all $i$ by choice of $L$, so by our induction
assumption $x^{a+a_i}y^{b+b_i}z^{p+c_i}F^{d+i}\in J$ for all $i$.
Hence $h\in J$.

Now take elements $j\in J$, $s\in R$, and show that $js\in J$. It
suffices to show this for a generating element $j=x^ay^bz^cF^d$ of
$J$. Let $s=\sum_i \eta_ix^{\A_i}y^{\beta_i}z^{\gamma_i}\in R$ for
some $\eta_i\in K$ and $\A_i, \beta_i, \gamma_i\in\N_0$. Then
$js=\sum_i\eta_ix^{a+\A_i}y^{b+\beta_i}z^{c+\gamma_i}F^d$. But for
all $i$, the summand
$g_i=\eta_ix^{a+\A_i}y^{b+\beta_i}z^{c+\gamma_i}F^d$ is in $J$ by
above, so $js=\sum_ig_i\in J$ as well and $J$ is an ideal.
\end{proof}

\begin{prop}\label{OrdinaryPowerBasisLemma2}
Let $r\in\N$ and $\sS_r$ and $\sT_r$ be as above.
Then $I^r=\langle x^ay^bz^cF^d | (a,b,c,d)\in\sS_r \rangle = \langle
x^ay^bz^cF^d | (a,b,c,d)\in\sT_r\rangle$.
\end{prop}

\begin{proof}
By Lemma \ref{OrdinaryPowerBasisLemma2Aux1}, the second equality is immediate.
Define $J:=\langle x^ay^bz^cF^d |
(a,b,c,d)\in\sS_r\rangle=\langle x^ay^bz^cF^d
|(a,b,c,d)\in\sT_r\rangle$ as above.

Since $I=(xy,xF,yF)$, we know $I^r$ is generated by
$$G=\M{(xy)^s(xF)^t(yF)^u}{s,t,u\in\N_0,s+t+u=r}.$$ Thus, as $J$ is an
ideal, $G\subseteq J$ implies $I^r\subseteq J$.
So take a generator $g=(xy)^s(xF)^t(yF)^u\in I^r$, i.e. pick
$s,t,u\in\N_0$ such that $s+t+u=r$. Then $u=r-s-t$ and
$g=(xy)^s(xF)^t(yF)^{r-s-t}=x^{s+t}y^{r-t}F^{r-s}$. But
$(s+t)+(r-t)=r+s\geq r$, and $(s+t)+(r-s)=r+t\geq r$ and
$(r-t)+(r-s)=r+(r-s-t)=r+u\geq r$, so $\min(s+t, r-t)+(r-s)\geq r$,
and $(s+t)+(r-t)+(r-s)=2r$, so $(a=s+t, b=r-t,c=0,d=r-s)\in\sS_r$ and
hence $g\in J$ as desired.

For the reverse containment, we take a basis element
$g=x^ay^bz^cF^d\in J$, where $(a,b,c,d)\in\sT_r$, and show that $g\in
I^r$. Without loss of generality, we may assume that $a\leq b$, and
since $g\in J$ if and only if $x^ay^bF^d\in J$ and $x^ay^bF^d\in
I^r$ implies $g\in I^r$, we may take $c$ to be 0.

(a) If $d\leq b-a$ and consequently $a+d\geq r$, then $b-d-a\geq 0$ and
we can write $g=(xy)^a(yF)^d\cdot y^{b-(d+a)}\in I^r$ since
$a+d\geq r$ and $b-a-d\geq 0$.

Now assume $b-a<d<b+a$. 

(b) If $b>r$, then we can write
$g=(xy)^a(yF)^{b-a}\cdot F^{d+a-b}\in I^r$ since
$a+(b-a)+(d+a-b)=d+a>b>r$. 

(c) If $b\leq r$, then $2(a+b)>a+b+d\geq 2r$
implies $a+b>r$ and therefore $a+b-r>0$. But it also implies
$(a+b+d)-r-d\geq 2r-r-d$, i.e. $a+b-r\geq r-d$. Therefore $a+b-r\geq
\max(r-d,0)$, so we can write
$g=(xy)^{\max(r-d,0)}(xF)^{r-b}(yF)^{(b+d)-r}\cdot
x^{a+b-r-\max(r-d,0)}y^{r-d-\max(r-d,0)}\in I^r$ since
$\max(r-d,0)+(r-b)+(b+d-r)\geq (r-d)+(r-b)+(b+d-r)\geq r$ and
$b+d-r\geq r-a\geq 0$ as $a+b+d\geq 2r$ and $r\geq b\geq a$. 

(d) Finally, assume $a+b\leq d$ and consequently $a+b\geq r$. Then
$d-a-b\geq 0$ and we can write $g=(xF)^a(yF)^b\cdot F^{d-a-b}\in
I^r$.

Therefore, we always have $g\in I^r$, and hence $J\subseteq I^r$.
\end{proof}

Combining Propositions \ref{SymPowerBasisProp} and \ref{OrdinaryPowerBasisLemma2} gives the following criterion for containment of $I^{(m)}$ in $I^r$.

\begin{lem}\label{CompleteSolution1}
For $m,r\in\N$, $I^{(m)}\nsubseteq I^r$ if and only if either $r>m$, or $r\leq m$ and there exists an element $x^ay^bF^d\in I^{(m)}$ such that $a+b+d<2r$.
\end{lem}

\begin{proof}
There are elements in $I^{(m)}$ such that $\min(a,b)+d=m$, where $a,b,d\in\N_0$ are as in the description of $I^{(m)}$ in 
Proposition \ref{SymPowerBasisProp}.
For example, $g=x^{\lceil\F{m}{2}\rceil} y^{\lfloor\F{m}{2}\rfloor}F^{\lceil\F{m}{2}\rceil}$ is in $I^{(m)}$.
Therefore, Propositions \ref{SymPowerBasisProp} and \ref{OrdinaryPowerBasisLemma2} say that $I^{(m)}\nsubseteq I^r$ if either $r>m$ (e.g. $g\notin I^r$ because $\lfloor\F{m}{2}\rfloor+\lceil\F{m}{2}\rceil=m<r$ and hence the conditions that $a+b\geq r$ and $\min(a,b)+d\geq r$ are not satisfied), or $r\leq m$ and there exists an element $x^ay^bF^d\in I^{(m)}$ such that $a+b+d<2r$.

Conversely, if $I^{(m)}\nsubseteq I^r$, then Propositions \ref{SymPowerBasisProp} and \ref{OrdinaryPowerBasisLemma2} say that there is a basis element $h$ of $I^{(m)}$ that violates at least one of the conditions for $I^r$.
If either the condition $a+b\geq r$ or the condition $\min(a,b)+d\geq r$ are violated, then $h\in I^{(m)}$ means $a+b\geq m$ and $\min(a,b)+d\geq m$ and hence $r>m$.
If $r\leq m$ and hence $h$ satisfies these two conditions,
then it has to violate the third condition, and thus $a+b+d<2r$.
\end{proof}

We want to find the resurgence
$$\rho(I)=\sup_{m,r}\M{\F{m}{r}}{I^{(m)}\nsubseteq I^r}.$$
In fact, we will exhibit a condition on $m$ and $r$ that is necessary and sufficient for $I^{(m)} \subseteq I^r$.
In other words, given $m$ we will find the largest $r$ such that we have the containment $I^{(m)}\subseteq I^r$.

\begin{thm}\label{CompleteSolution2}
We have $I^{(m)}\subseteq I^r$ if and only if $4r\leq 3m+1$.
In particular, the resurgence is $\rho(I)=\F{4}{3}$.
\end{thm}

\begin{proof}
Suppose that $m,r\in\N$ are such that $I^{(m)}\subseteq I^r$.
By Lemma \ref{CompleteSolution1}, this means that $m\geq r$.
If $4r>3m+1$, we show that there exists an element $x^ay^bF^d$ in $I^{(m)}$ such that $a+b+d<2r$, which is a contradiction and hence shows that $4r\leq 3m+1$ is required.
Consider two cases, $m$ even and $m$ odd.

If $m$ is even, then $a=b=d=\F{m}{2}$ satisfies the requirements
for $I^{(m)}$ as any two of the exponents add to $m$, but $a+b+d=\F{3}{2}m<2r-\F{1}{2}<2r$, so $x^{\F{m}{2}}y^{\F{m}{2}}F^{\F{m}{2}}\in I^{(m)}\setminus I^r$.
If $m$ is odd, then $a=d=\F{m+1}{2}$ and $b=\F{m-1}{2}$ satisfy the requirements for $I^{(m)}$ as any two of the exponents add to at least $m$, but $a+b+d=\F{3}{2}m+\F{1}{2}<2r$, so $x^{\F{m+1}{2}}y^{\F{m-1}{2}}F^{\F{m+1}{2}}$ is in $I^{(m)}$ but not in $I^r$.

Conversely, suppose that $m,r\in\N$ are such that $4r\leq 3m+1$, and therefore $r\leq m$.
We will now show that then necessarily $I^{(m)}\subseteq I^r$.
Let $x^ay^bF^d\in I^{(m)}$, so $a+b\geq m$ and $\min(a,b)+d\geq m$.
By symmetry, we may assume that $a\leq b$, so $a+d\geq m\geq r$.
Then also $b+d\geq m$, and therefore $(a+b)+(a+d)+(b+d)\geq 3m$, i.e. $2(a+b+d)\geq 3m\geq 4r-1$, which means $a+b+d\geq 2r-\F{1}{2}$.
Since $a,b,d,r\in\N_0$, we get $a+b+d\geq 2r$ and therefore $I^{(m)}\subseteq I^r$ by Lemma \ref{CompleteSolution1}.

Now, since $I^{(m)}\subseteq I^r$ if and only if $4r\leq 3m+1$, we get $I^{(m)}\nsubseteq I^r$ if and only if $4r >3m+1$, i.e. $I^{(m)}\nsubseteq I^r$ if and only if $\F{4}{3}-\F{1}{3r}> \F{m}{r}$.
Thus $\rho(I)\leq \F{4}{3}$ because $\F{m}{r}\geq \F{4}{3}$ implies $I^{(m)}\subseteq I^r$.
Given $r\in \{3n+1\}_{n\in\N}$, $m=\F{4r-1}{3}-1=\F{4(r-1)}{3}\in\N$ satisfies $4r>3m+1$ and hence $I^{(m)}\nsubseteq I^r$.
But then $\F{m}{r}=\F{4}{3}\K{\F{r-1}{r}}\leq\rho(I)$ for all $r$, so
$\rho(I)\geq\F{4}{3}$. Therefore $\rho(I)=\F{4}{3}$.
\end{proof}




\section{Consequences and Applications}\label{SectionConsandApps}

We can use our methods to obtain results on factoring symbolic powers.

\subsection{Almost Collinear Configuration}

Given a homogeneous ideal $J$, $\alpha(J)$ denotes the least degree $d$ for which $J_d \not= 0$.
In other words, $\alpha(J)$ denotes the degree in which the ideal $J$ begins.

\begin{lem}\label{LemAlphas}
Let $I$ be the ideal of $n+1$ almost collinear points, where $n\geq 3$.
Then  $\alpha(I^{(m)}) = \lceil m(2n-1)/n \rceil$ and $\alpha(I^r) = 2r$.
\end{lem}
\begin{proof}
Suppose $L_0$ is the line containing the $n$ collinear points $p_1,\cdots,p_n$, and suppose $L_i$ is a line containing $p_0$ and $p_i$ for each $1\leq i\leq n$.
Set $a = \lceil m(2n-1)/n \rceil$, and write $m = bn + r$, where $0\leq r < n$ (thus, $b = \lfloor m/n\rfloor$).
Then $a = \lceil 2m-(bn+r)/n\rceil = \lceil 2m - b-r/n\rceil = 2m-b$.
Finally, $L_0^{m-b} L_1^{b+1}\cdots L_{r}^{b+1} L_{r+1}^b \cdots L_n^b$ has degree $2m-b$ and vanishes at each $p_j$ to order at least $m$, which demonstrates that $\alpha(I^{(m)}) \leq \lceil 2m(n-1)/n\rceil$.

For an upper bound, consider $F = nL - (n-1)E_0 - E'$, where $E' = E_1 + E_2 + \cdots + E_n$. Set $E = E_0 + E'$.
Then $F$ is nef, as it meets each $L_i = L - E_0 - E_i$ and $E_0$ nonnegatively, and $F = L_1 + L_2 + \cdots + L_n + E_0$.
Thus, $F\cdot (\alpha L - mE) \geq 0$, and since $F\cdot (\alpha L - mE) = n\alpha - m(n-1)-mn \geq 0$ we see $\alpha \geq (2n-1)m/n$, and thus $\alpha = \lceil m(2n-1)/n\rceil$, which proves $\alpha(I^{(m)}) = \lceil m(2n-1)/n\rceil$.

As for the ordinary power, $\alpha(I) = 2$ immediately gives $\alpha(I^r) = 2r$.
\end{proof}

\begin{rem}
Notice that, if $A\in \field[x,y]$ and $\deg A = d$, then $A$ vanishes to order $d$ at $p_0$, as $I(p_0)^{(d)} = (x,y)^d$.
\end{rem}

\begin{prop}\label{PropSymPowersArePowers}
Given $I$ as above, $I^{(nt)} = (I^{(n)})^t$ for all $t\geq 1$.
Moreover, we have $$n = \min\setof{e}{I^{(et)} = (I^{(e)})^t\,\,\forall t\geq 1}.$$
\end{prop}
\begin{proof}
To prove the first statement, we use induction on $t$.
When $t=1$, the statement is clear.
For $t > 1$, write $(I^{(n)})^t = I^{(n)} (I^{(n)})^{t-1} = I^{(n)} I^{(n(t-1))}$.
Thus we need only show $I^{(n)} I^{(n(t-1))} = I^{(nt)}$.
An easy geometric argument yields the forward containment, so consider the reverse.

Consider $H_i y^j z^l \in I^{(nt)}$.
By Lemma \ref{SymPowerBasisLem}, $i+ln\geq n^2 t$ and $i+j\geq nt$, where the first inequality is equivalent to $\lfloor i/n\rfloor + l \geq nt$.
Assume $\lfloor i/n\rfloor \geq 1$ and $l \geq n-1$.
Then $H_i y^j z^l = x^e F^{\lfloor i/n\rfloor} z^l = (F z^{n-1})\left(x^e F^{\lfloor i/n\rfloor -1} z^{l-n+1}\right)$.
Notice that $F z^{n-1} \in I^{(n)}$; we claim $x^e F^{\lfloor i/n\rfloor -1} z^{l-n+1} = H_{i-n} y^j z^{l-n+1} \in I^{(n(t-1))}$.
We have by hypothesis that $i+j\geq nt$ and $i+ln \geq n^2 t$.
The latter inequality is equivalent to $(i-n)+(l-(n-1))n\geq (t-1)n^2$, and subtracting $n$ from both sides of the former shows that $(i-n)+j\geq n(t-1)$; this proves the claim in the case of $\lfloor i/n \rfloor \geq 1$ and $l \geq n-1$.

Suppose now that $\lfloor i/n\rfloor = 0$.
This means that $0\leq i < n$, and thus $H_i = x^i$, so $H_i y^j z^l = x^i y^j z^l$.
We have by hypothesis that $i+j\geq nt$ and $l = \lfloor i/n\rfloor +l \geq nt$.
Therefore, we can factor $x^i y^j z^l = (yz)^n (x^i y^{j-n} z^{l-n})$, where $(yz)^n \in I^{(n)}$ and $x^i y^{j-n} z^{l-n} \in I^{(n(t-1))}$.


Finally, suppose $l \leq n-2$.
Write $l = n-2-\delta$, where $0\leq \delta\leq n-2$.
We know $i+ln \geq n^2t$, so $i\geq n^2 t - ln = n^2 t - (n-2-\delta)n = n^2 (t-1) + (\delta+2)n$, and thus $b = \lfloor i/n \rfloor \geq n(t-1)+\delta+2$.
Set $\e = b - n(t-1) -\delta-2$.
Using these constraints on $i,j,l$, we can write $H_i y^j z^l = x^e F^b y^j z^l = (x^e y^j F^{n(t-1)}) F^\e F^{\delta+2} z^l$.
Since $F$ vanishes at each point, $F^{n(t-1)}$ vanishes to order $n(t-1)$ at each point, and thus $x^e y^j F^{n(t-1)} \in I^{(n(t-1))}$.
As $l = n-2-\delta$, we see $l+\delta+2 = n$; therefore, $F^{\delta+2} z^l \in (z,F)^n$.
Additionally, $\delta+2 \geq 1$, so $F \in (x,y)^n$ and thus $F^{\delta+2} z^l \in (x,y)^n \cap (z,F)^n = I^{(n)}$.
Therefore, when $H_i y^j z^l \in I^{(nt)}$ and $l \leq n-2$, we conclude $H_i y^j z^l \in I^{(n(t-1))} I^{(n)}$, which completes the proof of the first statement.


To see the second statement, recall the computation of $\alpha(I^{(m)})$ from Lemma \ref{LemAlphas}, and assume $e < n$.
We know that $\alpha(I^{(et)}) = \lceil et(2n-1)/n\rceil = \lceil t(2e-e/n)\rceil$ and $\alpha((I^{(e)})^t) = t \lceil e(2n-1)/n\rceil = t\lceil 2e-e/n\rceil = 2et$, so that when $t \geq n/e$ we have $\alpha(I^{(et)}) < \alpha((I^{(e)})^t)$.
Thus, the ideals cannot be equal for all $t\geq 1$.
\end{proof}

\begin{rem}\label{RemarkNoetherian}
As a result of Proposition \ref{PropSymPowersArePowers}, we can conclude that the symbolic power algebra $\oplus I^{(m)}$ is Noetherian.
This is a homogeneous version of Theorem 1.3 in \cite{Schenzel}.
\end{rem}

\subsection{Nearly-Complete Intersection Configuration}

Unlike most results we present here, the following lemma about $\A$ for nearly-complete intersections is dependent on the number of points.

\begin{lem}\label{AlphaOfILem}
For $m,r\in\N$, we have $\A(I)=2$, $\A(I^r)=2r$, and $\A(I^{(m)})=
\lceil \F{3m}{2}\rceil$ if $n=1$ and $2m$ if $n\geq 2$.
\end{lem}

\begin{proof}
Since $I=(xy,xF,yF)$ and $F$ has degree $n\geq 1$, we immediately
get $\A(I)=2$. Moreover, $I^r$ has generators of degree $r\A(I)$ and no generator of lesser degree,
hence $\A(I^r)=r\A(I)$ so $\A(I^r)=2r$.

To find $\A(I^{(m)})$, recall that $x^my^m\in I^{(m)}$ for any
$n\geq 1$, so $\A(I^{(m)})\leq 2m$. Also, for $n=1$, we have
$x^{\lceil\F{m}{2}\rceil}
y^{\lceil\F{m}{2}\rceil}F^{\lfloor\F{m}{2}\rfloor}=x^{\lceil\F{m}{2}\rceil}
y^{\lceil\F{m}{2}\rceil}z^{\lfloor\F{m}{2}\rfloor}\in I^{(m)}$,
which has degree $\lceil\F{3m}{2}\rceil\leq 2m$. Therefore,
$\A(I^{(m)})\leq 2m$ if $n\geq 2$ and $\A(I^{(m)})\leq
\lceil\F{3m}{2}\rceil$ if $n=1$. We will show that we also have the
other inequalities, i.e. $\A(I^{(m)})\geq 2m$ if $n\geq 2$ and
$\A(I^{(m)})\geq \lceil\F{3m}{2}\rceil$ if $n=1$. 

Consider an element $g=x^ay^bF^d\in I^{(m)}$. If $n=1$, then
$a+b\geq m$ and $\min(a,b)+d\geq m$ imply that $2(a+b+d)\geq 3m$, so
$\deg(g)= a+b+nd=a+b+d\geq \lceil\F{3m}{2}\rceil$, hence
$\A(I^{(m)})\geq \lceil\F{3m}{2}\rceil$ as $g$ was arbitrary. If
$n\geq 2$, then $\deg g=a+b+nd\geq 2m$ if $d\geq m$,
so we may assume that $d\leq m$. Then $a,b\geq m-d\geq 0$
as $\min(a,b)+d\geq m$. Thus $\deg(g)=a+b+nd\geq
(m-d)+(m-d)+nd=2m+(n-2)d\geq 2m$ and hence $\A(I^{(m)})\geq 2m$ as
$g$ was arbitrary.
\end{proof}

We also get a neat description for writing symbolic powers of ideals defining nearly-complete intersections as ordinary powers.


\begin{thm}\label{SymPowerSplitEven}
Let $\A, \beta\in\N$ with at least one of $\alpha,\beta$ even. Then, if $I$
defines a nearly-complete intersection, we have
$I^{(\A+\beta)}=I^{(\A)}I^{(\beta)}$.
\end{thm}

\begin{proof}
Assume without loss of generality that $\A$ is even. Let $g\in I^{(\A+\beta)}$.
In order to show $I^{(\A+\beta)}\subseteq I^{(\A)}I^{(\beta)}$, it is enough to show that all (ideal) generators of $I^{(\A+\beta)}$  
are in $I^{(\A)}I^{(\beta)}$, so by our description for $I^{(\A+ 
\beta)}$ from Proposition \ref{SymPowerBasisProp}, we may assume that $g=x^ay^bF^d$, where $a,b,d, s, t\in\N_0$ are such that $a 
\leq b$, $a+b=\A+\beta+s$, and $\min(a,b)+d=a+d=\A+\beta+t$.
We will consider two cases, (a) that $a\geq\F{\A}{2}$ and (b) that $a<\F{\A}{2}$. For the first case, we have the subcases (i) that $d \geq\F{\A}{2}$ and (ii) that $d<\F{\A}{2}$.

\begin{enumerate}
\item[(a)] Assume that $b\geq a\geq \F{\A}{2}$.

\begin{enumerate}

\item[(i)] Suppose we have $d\geq \F{\A}{2}$. Notice that $g$ can be written as $$g=
\underbrace{x^{\F{\A}{2}}y^{\F{\A}{2}}F^{\F{\A}{2}}}_{\in I^{(\A)}} 
\cdot \underbrace{x^{a-\F{\A}{2}}y^{b-\F{\A}{2}}F^{d-\F{\A}{2}}}_  {\in I^{(\beta)}}$$ where the first factor is in $I^{(\A)}$ because  
any two of the exponents add to $\A$, and the second factor is in  
$I^{(\beta)}$ because $(a-\F{\A}{2})+(b-\F{\A}{2})=a+b-\A=\beta+s$  
and $(\min(a,b)-\F{\A}{2})+(d-\F{\A}{2})=\beta+t$.

\item[(ii)] If $b\geq a\geq \F{\A}{2}$ but $d<\F{\A}{2}$, then $b\geq a\geq
\A-d$ as $b+d\geq a+d=\A+\beta+t\geq \A$. Also, $a-(\A-d)=\beta+t$  
and $b-(\A-d)=\beta+t+v$ for some
$v\in\N_0$. Then we can write $g=\underbrace{x^{\A-d}y^{\A-d}F^d}_ 
{\in I^{(\A)}}\cdot \underbrace{x^{\beta+t}y^{\beta+t+v}}_{\in I^ 
{\beta}}$ where the first factor is in $I^{(\A)}$ because $(\A-d)+d= 
\A$ and $2\A-2d>\A$, and the second factor is in $I^{(\beta)}$  
because $t,v\in\N_0$.

\end{enumerate}

\item[(b)] For the second case, assume that $a<\F{\A}{2}$. Then we have $a 
+b\geq \A>2a$ and $a+d\geq \A>2a$, so
$b\geq \A-a>a$ and $d\geq \A-a>a$.
Also, $b-(\A-a)=\beta+s$ and $d-(\A-a)=\beta+t$.
Therefore we can write $g=\underbrace{x^ay^{\A-a}F^{\A-a}}_{\in I^{(\A)}}\cdot \underbrace{y^{\beta+s}F^{\beta+t}}_{\in I^{(\beta)}}$  
where the first factor is in $I^{(\A)}$ because $a+(\A-a)=\A$ and $2 
(\A-a)>\A$, and the second factor is in $I^{(\beta)}$ because $s,t \in\N_0$.
Therefore we get $I^{(\A+\beta)}\subseteq I^{(\A)}I^{(\beta)}$.
\end{enumerate}

For the other direction, note that it suffices to take two (ideal)
generators $g\in I^{(\A)}$ and $h\in I^{(\beta)}$, and show that
$gh\in I^{(\A+\beta)}$. Again using \ref{SymPowerBasisProp}, we may  
assume that $g=x^{a_1}y^{b_1}F^{d_1}$ and $h=x^{a_2}y^{b_2}F^{d_2}$, where for $i=1,2$, we have $a_i, b_i, d_i, s_i, t_i\in\N_0$ such that
$a_1+b_1=\A+s_1$, $a_2+b_2=\beta+s_2$, $\min(a_1, b_1)+d_1=\A+t_1$,
and $\min(a_2, b_2)+d_2=\beta+t_2$. We immediately obtain
$(a_1+a_2)+(b_1+b_2)=\A+\beta+(s_1+s_2)$ and $\min(a_1+a_2,
b_1+b_2)+(d_1+d_2)\geq \min(a_1, b_1)+\min(a_2,
b_2)+(d_1+d_2)=\A+\beta+(t_1+t_2)$, so
$gh=x^{a_1+a_2}y^{b_1+b_2}F^{d_1+d_2}\in I^{(\A)}I^{(\beta)}$ as
desired.
\end{proof}

\begin{cor}\label{CorSymPowersArePowers2}
Let $r,s,t\in\N$. Then $I^{(2st)}=\K{I^{(2s)}}^t$ and
$I^{((2s+r)t)}=I^{(2st)}I^{(rt)}$.
\end{cor}

\begin{proof}
Both equations follow immediately from Theorem \ref{SymPowerSplitEven}.
\end{proof}

Also note that, as a result of Corollary \ref{CorSymPowersArePowers2}, $I^{(2t)} = (I^{(2)})^t$ and, as with Proposition \ref{PropSymPowersArePowers}, we can conclude that the symbolic power algebra $\oplus I^{(m)}$ is Noetherian.
We now refine our results on factoring symbolic powers of ideals of nearly-complete intersections.

\begin{lem}\label{SymPowerSplitAux}
Let $m\in\N$ be odd. Then $I^{(m)}I=\langle x^ay^bz^cF^d
| a,b,c,d\in\N_0, c<n, a+b\geq m+1, \min(a,b)+d\geq m+1, a+b+d\geq
\F{3m+1}{2}+2 \rangle$.
\end{lem}

\begin{proof}
Let $\sS:=\langle x^ay^bz^cF^d | a,b,c,d\in\N_0, c<n, a+b\geq m+1,
\min(a,b)+d\geq m+1, a+b+d\geq \F{3m+1}{2}+2 \rangle$. By Proposition 
\ref{SymPowerBasisProp},
$I^{(m)}I$ is generated by elements of the form
$g=\underbrace{x^{a_1}y^{b_1}z^{c_1}F^{d_1}}_{\in I^{(m)}}\cdot
\underbrace{x^{a_2}y^{b_2}z^{c_2}F^{d_2}}_{\in I}$, where for
$i=1,2$, $a_i,b_i,c_i,d_i\in\N_0$ are such that $c_1<n$, $c_2<n$,
$a_1+b_1\geq m$, $\min(a_1,b_1)+d_1\geq m$, $a_2+b_2\geq 1$,
$\min(a_2,b_2)+d_2\geq 1$, and $a_2+b_2+d_2\geq 2$. Note that also
$(a_1+b_1)+(a_1+d_1)+(b_1+d_1)=2(a_1+b_1+d_1)\geq 3m$, so
$a_1+b_1+d_1\geq \F{3m}{2}$. But $m$ is odd, so we actually have
$a_1+b_1+d_1\geq \F{3m+1}{2}$.

Thus any generator $g$ for $I^{(m)}I$ is of the form
$x^{a_1+a_2}y^{b_1+b_2}z^{c_1+c_2}F^{d_1+d_2}$ where
$(a_1+a_2)+(b_1+b_2)\geq m+1$, $\min(a_1+a_2, b_1+b_2)+(d_1+d_2)\geq
\min(a_1,b_1)+\min(a_2,b_2)+(d_1+d_2)\geq m+1$,
$(a_1+a_2)+(b_1+b_2)+(d_1+d_2)\geq \F{3m+1}{2}+2$, and, by Proposition 
\ref{LemXSpan2}, we may assume $c_1+c_2<n$. Therefore, $g\in\sS$ and
hence $I^{(m)}I\subseteq \sS$.

For the other containment, let $g=x^ay^bz^cF^d\in\sS$, where we may
assume that $c=0$. Let $s,t,u\in\N_0$ be such that $a+b=m+1+s$,
$a+d=m+1+t$, and $b+d=m+1+u$. Then
$a+b+d=\F{3m+3}{2}+\F{s+t+u}{2}\geq \F{3m+1}{2}+2$ implies that
$\F{s+t+u}{2}\geq 1$, and $a+b=m+1+s$ implies that $a\geq 1$ or
$b\geq 1$. 

Say $a\geq 1$.
If $d=0$, then $a-1=m+t$ and $b-1=m+u$,
so $g=\underbrace{x^{a-1}y^{b-1}}_{\in
I^{(m)}}\cdot\underbrace{xy}_{\in I}\in I^{(m)}I$.  

So suppose that $d\geq 1$.

(a) If $s\geq 1$, we get $(a-1)+d=m+t$, $(b-1)+d=m+u$,
and $(a-1)+(b-1)=m+(s-1)\geq m$, so
$g=\underbrace{x^{a-1}y^{b-1}F^d}_{\in I^{(m)}}\underbrace{xy}_{\in
I}\in I^{(m)}I$.

(b) If $s=0$ and $u,t\geq 1$, then $(a-1)+b=m$,
$(a-1)+(d-1)=m+(t-1)\geq m$, and $(b-1)+(d-1)=m+(u-1)\geq m$, so
$g=\underbrace{x^{a-1}y^bF^{d-1}}_{\in I^{(m)}}\underbrace{xF}_{\in
I}\in I^{(m)}I$.

(c) Finally, suppose $s=0$ and either of $u$ or $t$ is
also zero, say $t=0$. Then $u\geq 2$ as $s+t+u\geq 2$, and $b=d\geq
1$. Thus we get $a+(b-1)=m$, $a+(d-1)=m$, and
$(b-1)+(d-1)=m+(u-2)\geq m$, so
$g=\underbrace{x^ay^{b-1}F^{d-1}}_{\in I^{(m)}}\underbrace{yF}_{\in
I}\in I^{(m)}I$.

Therefore $g\in I^{(m)}I$ and hence $\sS\subseteq I^{(m)}I$.
\end{proof}

The condition that $\A$ or $\beta$ be even in Theorem \ref{SymPowerSplitEven} is necessary, as we see in Theorem \ref{SymPowerSplitOdd}.

\begin{thm}\label{SymPowerSplitOdd}
Let $\A,\beta\in\N$ both be odd. Then
$I^{(\A+\beta)}\supsetneq I^{(\A)}I^{(\beta)}$.
\end{thm}

\begin{proof}
By Corollary \ref{CorSymPowersArePowers2}, for any odd number $m\in\N$, we have $I^{(m)}=I^{(m-1+1)}=I^{(m-1)}I^{(1)}=I^{(m-1)}I$ since $m-1$ is even.
Notice that the parity of $\A$ was irrelevant for the second part of the proof for Theorem \ref{SymPowerSplitEven}, so we still have $I^{(\A+\beta)}\supseteq I^{(\A)}I^{(\beta)}$.
If $\A=2k+1$ and $\beta=2l+1$ for some $k,l\in\N_0$, then $I^{(\A)}=\K{I^{(2)}}^kI$ and $I^{(\beta)}=\K{I^{(2)}}^lI$ by Corollary \ref{CorSymPowersArePowers2}.
Then $\A+\beta-2$ is even and thus $I^{(\A)}I^{(\beta)}=\K{I^{(2)}}^{k+l}I^2=I^{2k+2l}I^2= \K{I^{\A+\beta-2}I}I=I^{(\A+\beta-1)}I$.
However, $\A+\beta-1$ is odd, so we cannot simplify $I^{(\A+\beta-1)}I$ any further.
By Lemma \ref{SymPowerSplitAux}, $I^{(\A+\beta-1)}I=\langle x^ay^bz^cF^d | a,b,c,d\in\N_0, c<n, a+b\geq \A+\beta, \min(a,b)+d\geq \A+\beta, a+b+d\geq \F{3(\A+\beta)}{2}+1 \rangle$.
But $g=x^{\F{\A+\beta}{2}}y^{\F{\A+\beta}{2}}F^{\F{\A+\beta}{2}}\in I^{(\A+\beta)}$ and $3\K{\F{\A+\beta}{2}}<\F{3(\A+\beta)}{2}+1$, so $g$ does not satisfy the condition that $a+b+d\geq \F{3(\A+\beta)}{2}+1$ and therefore $g$ is contained in $I^{(\A+\beta)}$ but not in $I^{(\A)}I^{(\beta)}$.
\end{proof}

\begin{cor}\label{SymPowerSplitComplete}
For $m\in\N$, we have $I^{(m)}=\K{I^{(2)}}^{\F{m}{2}}$
if $m$ is even, and $I^{(m)}=\K{I^{(2)}}^{\F{m-1}{2}}I$ if $m$ is
odd.
\end{cor}

\begin{proof}
The even case follows by Corollary \ref{CorSymPowersArePowers2} and the odd case
by Corollary \ref{CorSymPowersArePowers2} and Theorem \ref{SymPowerSplitOdd}.
\end{proof}


\subsection{Common Results}

This section contains applications of the results of Section \ref{SecMainResults} which are true for both almost collinear points and nearly-complete intersections.
In particular, we verify two conjectures of \cite{BCH}. 

\begin{thm}
If $I$ defines almost collinear points or a nearly-complete
intersection, then $I^{2r} \subseteq \sM^r I^r$ and $I^{(2r-1)} \subseteq
\sM^{r-1} I^r$, where $\sM = (x,y,z)$ is the irrelevant maximal
ideal.
\end{thm}
\begin{proof}
\textbf{Nearly-complete intersection: } If $I$ defines a nearly-complete intersection, notice that
$I^{(2)}=(x^2y^2, x^2F^2, xyF, y^2F^2)$ and $$\sM
I=(x,y,z)(xy,xF,yF)=(x^2y, x^2F, xyF, xy^2, y^2F, xyz, xFz,yFz),$$
and thus $I^{(2)}\subseteq \sM I$. Then Corollaries
\ref{CorSymPowersArePowers2} and  \ref{SymPowerSplitComplete} give
$I^{(2r)}=\K{I^{(2)}}^r\subseteq \K{\sM I}^r=\sM^rI^r$ and
$I^{(2r-1)}=\K{I^{(2)}}^{r-1}I\subseteq \K{\sM
I}^{r-1}I=\sM^{r-1}I^r$.

\textbf{Almost collinear points: } Now assume that $I$ defines $n+1$ almost collinear points, and consider $H_i y^j z^l\in I^{(2r)}$.
By Lemma \ref{SymPowerBasisLem}, this means that $i+j\geq 2r$ and $i+ln\geq 2rn$ (equivalently, $\lfloor i/n\rfloor + l \geq 2r$).
Suppose $l \geq r$.
Then we may write $H_i y^j z^l = z^r G_1 G_2$, where $G_1$ is a form in $x$ and $y$ of degree $r$ dividing $H_i y^j$ (such a form exists since $i+j\geq 2r$ and $H_i y^j = x^e L_1^a \cdots L_n^a y^j$, where the $L$'s are linear factors, $F = L_1 \cdots L_n$, and $i = an+e$), and $G_2 = H_i y^j z^l/z^r G_1$.
Since $I = (xz,yz,F)$, it follows that $z^r G_1 \in I^r$, and since $i+j\geq 2r$, $G_2$ has degree at least $r$, and hence $G_2 \in \sM^r$.
Therefore, $H_i y^j z^l \in \sM^r I^r$.

Suppose now that $l < r$; then there is an integer $\delta$ with $l = r -\delta$ satisfying $1\leq\delta\leq r$.
Then $2nr \leq i + nl = i + n(r-\delta) = i+nr - n\delta$, and subtracting $nr$ gives $nr \leq i - n\delta$.
Thus, $n(r+\delta)\leq i$ and hence there is an integer $\e$ such that $i = n(r+\delta) + \e$.
Then $H_i y^j z^l = F^r H_{n\delta+\e} y^j z^l = F^r H_{n\delta+\e} y^j z^{r-\delta}$.
Since $F^r \in I^r$ and $\deg (H_{n\delta+\e} y^j z^{r-\delta}) = n\delta+\e+j+r-\delta = (n-1)\delta + \e + j+ r \geq r$ (since each of the summands is nonnegative), we have that $H_{n\delta+\e} y^j z^{r-\delta} \in \sM^r$.
Therefore, $I^{(2r)} \subseteq \sM^r I^r$.

Now we consider the other containment for the ideal defining almost collinear points.
As before, if $H_i y^j z^l \in I^{(2r-1)}$, Lemma \ref{SymPowerBasisLem} implies that $i+j\geq 2r-1$ and $i+ln \geq (2r-1)n$.
If $l\geq r$, then $H_i y^j z^l = z^r z^{l-r} H_i y^j$.
Note that we can factor $H_i y^j = L_1 \cdots L_{i+j}$ as a product of linear factors, and that each linear factor is a polynomial in $\field[x,y]$.
Since $i+j\geq 2r-1$, we can collect $r$ of these linear factors and call their product $G = L_1 \cdots L_r$.
Therefore, we can write $H_i y^j z^l = (z^r G) z^{l-r} L_{r+1} \cdots L_{i+j}$; since $z^r G \in I^r$ and $l-r+i+j -r \geq l-2r+2r-1 = l-1 \geq r-1$, we find that $H_i y^j z^l \in \sM^{r-1} I^r$.

Suppose now that $r > l$, i.e., there is an integer $\delta$ such that $l = r-1 - \delta$ and $0\leq \delta \leq r-1$.
Then $(2r-1)n \leq i + ln = i+(r-1-\delta)n$, so $(r+\delta)n\leq i$.
Thus, there is an integer $\e$ such that $i = (r+\delta)n + \e$.
We can therefore write $H_i y^j z^l = F^r H_{\delta n+\e} y^j z^l$, and since $F\in I$, we know $F^r \in I^r$.
Since $H_{\delta n + \e} y^j z^{r-1-\delta}$ has degree $\delta n + \e + j + r - 1 - \delta = \delta (n-1) + \e + j + r-1 \geq r-1$, it follows that $H_{\delta n + \e} y^j z^{r-1-\delta} \in \sM^{r-1}$.
Thus, we conclude $I^{(2r-1)} \subseteq \sM^{r-1} I^r$.
\end{proof}

\begin{thm}
Let $\sM = (x,y,z)\subseteq \field[x,y,z]$, where $\field$ is a
field of characteristic 0. If $I$ defines almost collinear points or
a nearly-complete intersection, then $I^{(t(m+1))} \subseteq \sM^t
(I^{(m)})^t$ and $I^{(t(m+1)-1)} \subseteq \sM^{t-1} (I^{(m)})^t$.
\end{thm}
\begin{proof}
\textbf{Nearly-complete intersection: } By Corollary
\ref{CorSymPowersArePowers2} (first with $s=1$, then with $r=t=1$),
Lemma 2.4 and Proposition 2.3 of \cite{BCH} apply to give the
results for nearly-complete intersections.

\textbf{Almost collinear points: } For ideals $I$ defining almost collinear points, Proposition \ref{PropSymPowersArePowers} demonstrates that $I^{(2j)} \not= (I^{(2)})^j$ for all $j\geq 1$, so Proposition 2.3 of \cite{BCH} does not apply.
Instead, we use the $\field$-basis developed above, and first consider the containment $I^{(t(m+1))} \subseteq \sM^t (I^{(m)})^t$.

We wish to factor a basis element $H_i y^j z^l \in I^{(t(m+1))}$ into a product of a form of degree $t$ and a product of $t$ forms, each of which vanishes to order $m$ on the set of $n+1$ points.
The symbolic power basis inequalities (see Lemma \ref{SymPowerBasisLem}), in this setting, are
\begin{equation}\label{FactorConj1zFInequality}
i+nl \geq nt(m+1)
\end{equation}
and
\begin{equation}\label{FactorConj1xyInequality}
i+j \geq t(m+1).
\end{equation}

If $l=0$, then (\ref{FactorConj1zFInequality}) becomes $i \geq nt(m+1)$, which means $H_i$ has a factor of $F^{tm+t}$.
It is clear that $F^t \in \sM^t$, as $\deg F^t = nt \geq t$, and $F^m \in I^{(m)}$, as $F$ vanishes at each of the $n+1$ points.
Thus, $H_i y^j z^l \in \sM^t (I^{(m)})^t$.

Now assume $l\geq 1$.

If $1 \leq l < t$, then $l+\gamma = t$ for some $\gamma \geq 1$.
We see that (\ref{FactorConj1zFInequality}) becomes $i+n(t-\gamma) \geq nt(m+1)$, and thus $i \geq n(tm+\gamma)$.
Then $F^{tm+\gamma}$ is a factor of $H_i$; as before, $F^{tm} \in (I^{(m)})^t$, and, as $l+\gamma = t$, $\deg z^l F^\gamma = l + n\gamma \geq l+\gamma = t$, so $z^l F^\gamma \in \sM^t$, whence $H_i y^j z^l \in \sM^t (I^{(m)})^t$.

If $l \geq mt$, write $l = mt + \gamma$.
Recall that $F = L_1 L_2 \cdots L_n$, where $L_i$ is a linear form vanishing at $p_0$ and $p_i$, for $1\leq i \leq n$.
Set $i = nb+e$, where $0 \leq e < n$.
Thus, we can factor $H_i y^j z^l = x^e F^b y^j z^l = x^e L_1^b L_2^b \cdots L_n^b y^j z^{mt} z^{l-mt}$.
If $j \geq t$, then $y^t \in \sM^t$, so write $H_i y^j z^l = y^t (x^e L_1^b L_2^b \cdots L_n^b y^{j-t} z^{mt} z^{l-mt})$.
As $z^m$ vanishes to order $m$ at the $n$ collinear points, we may group the linear factors of $H_i y^{j-t}$ as $G_1 G_2 \cdots G_t G$, where $\deg G_d = m$ and $\deg G = i+j-t -mt \geq 0$; then $G_d z^m \in I^{(m)}$ for every $d$, $1\leq d \leq t$, so $H_i y^j z^l \in \sM^t (I^{(m)})^t$.
If, on the other hand, $j < t$, set $\delta = t-j$.
Then (\ref{FactorConj1xyInequality}) becomes $i\geq tm + \delta$, so we again factor $H_i = G_1 G_2 \cdots G_t G$, where $\deg G_d = m$, and $\deg G = i-tm \geq \delta$.
Since $\delta + j = t$, $G y^j \in \sM^t$, and, as $G_d$ vanishes to order $m$ at $p_0$, $G_d z^m \in I^{(m)}$ for every $d$, $1\leq d \leq t$.
Thus, $H_i y^j z^l \in \sM^t (I^{(m)})^t$.

Finally, suppose $l$ satisfies $st \leq l < (s+1)t$, where $1\leq s\leq m-1$, and write $l = (s+1)t - \gamma$, $1\leq \gamma \leq t$.
Then (\ref{FactorConj1zFInequality}) becomes $i \geq n(t(m-s)+\gamma)$, so $H_i$ has a factor of $F^{t(m-s)+\gamma}$.
Notice that $F^{m-s} z^s$ vanishes to order $m$ at each of the $n$ collinear points.
We consider two cases: $j\geq t$ and $j < t$.

\textbf{Case 1: } Assume $j \geq t$.
Then it is obvious that $y^t \in \sM^t$.
If $n(m-s)\geq m$, then $F^{m-s}$ vanishes to order $m$ at $p_0$, and $F^{m-s} z^s \in I^{(m)}$, which proves that $H_i y^j z^l \in \sM^t (I^{(m)})^t$.

Suppose now that $n(m-s) < m$.
Set $\delta = j-t$; then $\delta \geq 0$.
Since $i+j\geq t(m+1)$, we know $i+t+\delta \geq t(m+1)$, whence $i+\delta \geq tm$.
This means that we can factor $H_i y^\delta$ into a product of $t$ factors, each vanishing at $p_0$ to order $m$; say $H_i y^\delta = G_1 G_2 \cdots G_t\cdot G$, where $\deg G_d = m$ for $1\leq d \leq t$, and $\deg G = i+\delta - mt$.
Our aim is to do this in such a way so that each $G_d$, when multiplied by a particular power of $z$, will vanish to order $m$ at each of the $n+1$ points.
Note that $i\geq n(t(m-s)+\gamma)$.
Then $F^{t(m-s)+\gamma}$ divides $H_i$, and so $H_i$ has $t$ factors of $F^{m-s}$.
Define $G_1 = F^{m-s} Q_1$, where $Q_1$ is a product of $m-n(m-s)$ linear factors of $H_iy^\delta/F^{t(m-s)}$.
Now recursively define, for each $d$ satisfying $2\leq d\leq t$, $G_d = F^{m-s} Q_d$, where $Q_d$ is a product of $m-n(m-s)$ linear factors of $\frac{H_i y^\delta}{F^{t(m-s)} Q_1 Q_2 \cdots Q_{d-1}}$ (note that this is possible, as $F^{t(m-s)}|H_i$, $i+\delta \geq tm$, each $G_d$ has degree $m$, and the factors $Q_d$ are distinct and chosen so that $Q_d | H_i y^\delta$).
Then $G_d$ vanishes to order $m$ at $p_0$ by construction.
Notice that $G_d z^s \in I^{(m)}$ as $\deg G_d = m$ (and thus vanishes to order $m$ at $p_0$), and $F^{m-s}z^s$ vanishes to order $m$ at $p_1,p_2,\ldots,p_n$.
Therefore, $H_i y^j z^l = z^t G z^{t-\gamma} \prod\limits_{d=1}^t G_d z^s \in \sM^t (I^{(m)})^t$.

\textbf{Case 2: } Now suppose $j < t$, and set $\delta = t-j$.
Recall that $l = (s+1)t - \gamma$, $1\leq \gamma \leq t$ and $1\leq s \leq m-1$ and (\ref{FactorConj1zFInequality}) becomes $i \geq n(t(m-s)+\gamma)$, so $H_i$ has a factor of $F^{t(m-s)+\gamma}$.
It is clear that $F^{m-s} z^s$ vanishes to order $m$ at $p_1,p_2,\ldots,p_n$, as both $F$ and $z$ vanish once at each of the $n$ points.
Moreover, if $n(m-s) \geq m$, $F^{m-s}$ vanishes to order $m$ at $p_0$, so $F^{t(m-s)}z^{st} \in (I^{(m)})^t$.
Since $F^\gamma z^{l-st} = F^\gamma z^{t-\gamma}$ has degree $n\gamma + t -\gamma \geq t$ (since $n\geq 3$), $F^\gamma z^{t-\gamma} \in \sM^t$, and thus $H_i y^j z^l \in \sM^t (I^{(m)})^t$.

Suppose instead that $n(m-s) < m$ and recall that (\ref{FactorConj1zFInequality}) becomes $i \geq n(t(m-s)+\gamma)$ and (\ref{FactorConj1xyInequality}) becomes $i \geq mt + \delta$.
Define $A_1 = F^{m-s} B_1$, where $B_1$ is a product of $m-n(m-s)$ linear factors of $H_i/F^{t(m-s)}$.
Recursively, for $d$ satisfying $1 < d \leq t$, define $A_d = F^{m-s} B_d$, where $B_d$ is a product of $m-n(m-s)$ linear factors of $\dfrac{H_i}{F^{t(m-s)} B_1 B_2 \cdots B_{d-1}}$.
(Note that we can do this, as $\deg H_i = i$ and $\deg A_1 A_2 \cdots A_t = mt$; since $F^{t(m-s)}|H_i$ and $F^{t(m-s)}| A_1 \cdots A_t$, and the other factors of $A_1 \cdots A_t$ are linear factors of $H_i$ [enough linear factors exist, since $i \geq mt$], it follows that $A_1 \cdots A_t | H_i$.)
Then $A_d$ is a form in $x$ and $y$ only of degree $n(m-s) + m - n(m-s) = m$, so $A_d$ vanishes to order $m$ at $p_0$.
Moreover, $A_d z^s$ vanishes to order $m$ at each of the $n+1$ points, so $A_d z^s \in I^{(m)}$.
Let $A$ be the form satisfying $H_i = (A_1 A_2 \cdots A_t) A$; then $\deg A = i - mt \geq \delta$, so $\deg A y^j \geq \delta + j = t$, so $A y^j \in \sM^t$.
Thus, $A y^j \in \sM^t$ and $A_1 A_2 \cdots A_t z^l \in (I^{(m)})^t$, so $H_i y^j z^l \in \sM^t (I^{(m)})^t$.

The containment $I^{(t(m+1)-1)} \subseteq \sM^{t-1} (I^{(m)})^t$ follows similarly.
\end{proof}


\bibliography{PowersOfIdeals}{}
\bibliographystyle{alphanum}

\end{document}